\documentclass[10pt]{article}

\usepackage{amsthm,amsmath,amssymb}
\usepackage{upgreek}

\usepackage{geometry}

\usepackage{xcolor}
\usepackage{hyperref}
\usepackage{enumitem}

\newcommand{\tl}{\mathtt{1}} 
\newcommand{\tO}{\mathtt{0}}
\newcommand{\occ}[2]{N^{#1}(#2)}
\newcommand{\occempty}[1]{N^{#1}}
\newcommand{\fl}[1]{\left\lfloor #1 \right\rfloor}
\newcommand{\N}{\mathbb{N}}
\newcommand{\Z}{\mathbb{Z}}
\newcommand{\R}{\mathbb{R}}
\newcommand{\C}{\mathbb{C}}
\newcommand{\dens}{\operatorname{dens}}

\DeclareMathOperator{\e}{e}     % exp(ix)

\newtheorem{customproposition}{Proposition}
\newenvironment{customprop}[1]
  {\renewcommand\thecustomproposition{#1}\customproposition}
  {\endcustomproposition}

\newtheorem{theorem}{Theorem}[section]
\newtheorem{proposition}[theorem]{Proposition}
\newtheorem{corollary}[theorem]{Corollary}
\newtheorem{lemma}[theorem]{Lemma}
\newtheorem{conjecture}[theorem]{Conjecture}

\theoremstyle{definition}
\newtheorem*{remark}{Remark}

\numberwithin{equation}{section}

\title{On the behavior of binary block-counting functions under addition}

\author{Bartosz Sobolewski\\Jagiellonian University in Krak\'{o}w, Poland\\
Montanuniversit{\"a}t Leoben, Austria \\
\url{bartosz.sobolewski@uj.edu.pl}
\date{}
}

\begin{document}
\maketitle
\begin{abstract}
Let $\mathsf{s}(n)$ denote the sum of binary digits of an integer $n \geq 0$. In the recent years there has been interest in the behavior of the differences $\mathsf{s}(n+t)-\mathsf{s}(n)$, where $t \geq 0$ is an integer. In particular, Spiegelhofer and Wallner showed that for $t$ whose binary expansion contains sufficiently many blocks of $\tl$s the inequality $\mathsf{s}(n+t) -\mathsf{s}(n) \geq 0$ holds for $n$ belonging to a set of asymptotic density $>1/2$, partially answering a question by Cusick. Furthermore, for such $t$ the values $\mathsf{s}(n+t) - \mathsf{s}(n)$ are approximately normally distributed.

In this paper we consider a natural generalization to the family of block-counting functions $\occempty{w}$, giving the number of occurrences of a block of binary digits $w$ in the binary expansion. Our main result show that for any $w$ of length at least $2$ the distribution of the differences $\occ{w}{n+t} - \occ{w}{n}$ is close to a  Gaussian when $t$ contains many blocks of $\tl$s in its binary expansion. This extends an earlier result by the author and Spiegelhofer for $w=\tl \tl$.
\end{abstract}

\renewcommand{\thefootnote}{\fnsymbol{footnote}}
% 	11A63   	Radix representation; digital problems
%   05A20     Combinatorial inequalities
% 11BO5 Density, gaps, topology
%   05A16     Asymptotic enumeration
%   11T71     Algebraic coding theory; cryptography

\section{Introduction and main result} \label{sec:intro}
The properties of binary expansions of integers is a commonly studied topic in number theory and theoretical computer science. Although arithmetic operations on such expansions are easily performed, their large-scale behavior is far from being fully understood due to propagation of carries. A problem which showcases this well is a deceptively simple question concerning the binary sum of digits $\mathsf{s}$, asked by Cusick in 2011 (and later ``promoted'' to a conjecture). Here and in the sequel we use the convention $\N = \{0,1,\ldots\}$. 

\begin{conjecture}[Cusick]
For all $t \in\N$ the natural density 
$$   c_t  = \dens \{n \in \N: \mathsf{s}(n+t) \geq \mathsf{s}(n)   \}$$
satisfies $c_t > 1/2$.
\end{conjecture}

 It can be verified that $c_t$ is well-defined, as a consequence of results B\'{e}sineau \cite[Lemme 1]{Besineau1972} (see also \cite[Lemma 3]{DrmotaKauersSpiegelhofer2016}).  Apart from being interesting in itself, the conjecture and the overall behavior of the differences $s(n+t)-s(n)$ are closely connected with other important problems. In particular, Cusick's conjecture is related to the Tu-Deng conjecture \cite{TuDeng2011} in cryptography. The latter conjecture, in fact, implies the former and holds almost surely, as proved in \cite{SpiegelhoferWallner2019}. Furthermore, the $2$-adic valuation $\nu_2$ of binomial coefficients in column $t$ of Pascal's triangle satisfies
$$ \mathsf{s}(n+t) - \mathsf{s}(n) = \mathsf{s}(t) - \nu_2\left( \binom{n+t}{t} \right),  $$
which is a quick corollary of Legendre's formula \cite{Legendre1830} for $\nu_2(m!)$. Therefore, the study of Cusick's conjecture may yield results on $2$-divisibility of binomial coefficients, and vice versa. For more details and other problems related to the conjecture see \cite[Section 3]{DrmotaKauersSpiegelhofer2016}.

Although Cusick's conjecture remains open so far, in the recent years there has been significant progress towards proving its validity.
It was verified numerically for $t<2^{30}$ by Drmota, Kauers and Spiegelhofer \cite{DrmotaKauersSpiegelhofer2016}.  In the same paper they showed that for any $\varepsilon > 0$  the inequality $1/2 < c_t < 1/2+\varepsilon$ holds for almost all $t$ in the sense of natural density. A central-limit type result on the distribution of  $\mathsf{s}(n+t)-\mathsf{s}(n)$ was established by Emme and Hubert \cite{EmmeHubert2019}, where $t$ is randomly chosen from $\{0,1,\ldots, 2^k-1\}$ and $k \to \infty$ (see also \cite{EmmeHubert2019,EmmePrikhodko2017}). Spiegelhofer \cite{Spiegelhofer2022} obtained a lower bound $c_t > 1/2-\varepsilon$ for all $t$ whose binary expansion of $t$ has sufficiently many maximal blocks of $1$s (improving his result from \cite{Spiegelhofer2019}). All these results were substantially improved in a recent paper of Spiegelhofer and Wallner \cite{SpiegelhoferWallner2023}. They showed that for each $t$ whose binary expansion contains $N \geq N_0$ maximal blocks of $\tl$s (where $N_0$ can be made explicit), the following statements are true: 
\begin{enumerate}
    \item The inequality $c_t > 1/2$ holds;
    \item For each $j \in \Z$ we have an asymptotic expression
    \begin{equation} \label{eq:s_asymptotic}
        \dens \{n \in \N: \mathsf{s}(n+t) - \mathsf{s}(n)  = j  \} =\frac{1}{\sqrt{2\pi \kappa(t)}} \exp \left(- \frac{j^2}{2 \kappa(t)} \right) + O(N^{-1} (\log N)^4),
    \end{equation}
     as $t \to \infty$, where $\kappa(t)$ denotes the variance of the corresponding distribution and satisfies a simple recursion.
\end{enumerate}
The first result is particularly important, as it reduces the task to proving that $c_t > 1/2$ for $t$ such that $N < N_0$. %Moreover, from  \eqref{eq:s_asymptotic} one can, in particular, deduce that $c_t$ is arbitrarily close to $1/2$ when $N$ is sufficiently large (see \cite[Corollary 1.3]{SpiegelhoferWallner2023}). 
However, as Cusick himself stated in private communication, the remaining case $N < N_0$ is the hard one. Indeed, for small $N$ the approximation by a Gaussian distribution is apparently not precise enough, and thus other methods need to be developed.

% whose main result is the following, where in the statement $M$.

% \begin{theorem}[{\cite[Theorem 1.1]{SpiegelhoferWallner2023}}] \label{thm:SW23_main}
%     There exists a constant $M_0$ with the following property: if the natural number $t$ has at least $M_0$ maximal blocks of $1$s in its binary expansion, then $c_t>1/2$.
% \end{theorem}

% In the same paper they also proved that for every $t$ having sufficiently many blocks of $1$s in its binary expansion, the distribution of the values $s(n+t)-s(n)$ is approximately Gaussian.

% \textcolor{red}{zobaczyc czy sie wypowiedz nie zmienila w ostatecznej wersji}
% \begin{theorem}[{\cite[Theorem 1.2]{SpiegelhoferWallner2023}}] \label{thm:SW23_Gaussian}
% For integers $t \geq 1$, let us define
% $$  \kappa_2(1) = 2; \qquad \kappa_2(2t)=\kappa_2(t); \qquad \kappa_2(2t+1) = \frac{\kappa_2(t)+\kappa_2(t+1)}{2}+1. $$
% If the positive integer$t$ has $M$ maximal blocks of $1$s in its binary expansion, and $M$ is larger than some constant $M_0$, then we have
% $$   \dens \{n \in \N: s(n+t) - s(n)  = j  \} =\frac{1}{\sqrt{2\pi \kappa_2(t)}} \exp \left(- \frac{j^2}{2 \kappa_2(t)} \right) + O(M^{-1} (\log M)^4) $$
% for all integers $j$. The multiplicative constant in the error term can be made explicit.
% \end{theorem}

% Cusick's conjecture is also closely linked to the Tu-Deng Conjecture

% \textcolor{red}{cos o Tu-Deng}

% and the behavior of the differences $s_2(n+t) - s_2(n)$ are also closely linked to other 

% \textcolor{red}{cos o binomialach}

While the full conjecture seems out of reach for the moment, we may as well ask whether the results obtained so far hold for a more general class of functions describing radix representations of integers. There are a few natural directions to consider, one of them being an extension to an arbitrary base $b \geq 2$. For the base-$b$ sum of digits $\mathsf{s}_b$, Hosten, Janvresse and de la Rue \cite{HostenJanvresseRue2024} proved that $\mathsf{s}_b(n+t)-\mathsf{s}_b(n)$ again satisfies a central-limit type result. Moreover, they estimated the error of approximation of the corresponding cumulative distribution function (after scaling) by a Gaussian. Interestingly, in the binary case neither this result, nor the expression \eqref{eq:s_asymptotic} seems to imply the other.

Another possible generalization is concerned with other patterns in integer expansions. In this direction, the author and Spiegelhofer \cite{SobolewskiSpiegelhofer2023} proved an analogue of the asymptotic formula \eqref{eq:s_asymptotic}, where the binary sum of digits is replaced with the function $\mathsf{r}$, counting the occurrences of the block $\tl \tl$ in the binary expansion. More precisely, they obtained for each $j \in \Z$ the asymptotic expression (formulated in an equivalent way):
\begin{equation} \label{eq:r_asymptotic}
    \dens \{n \in \N: \mathsf{r}(n+t) - \mathsf{r}(n)  = j  \} =\frac{1}{\sqrt{2\pi v_t}} \exp \left(- \frac{j^2}{2 v_t} \right) + O(N^{-1} (\log N)^2),
\end{equation}  
where $N$ has the same meaning as before, and $v_t$ denotes the variance of the corresponding distribution. Note that here the error term is  better by factor of $(\log N)^2$ in comparison with \eqref{eq:s_asymptotic}, and the same improvement should be possible there.

% \begin{theorem}[{\cite[Theorem ]{SobolewskiSpiegelhofer2023}}] \label{thm:SS23}
%     b
% \end{theorem}

In the present paper we continue this particular line of research and extend the formulas  \eqref{eq:s_asymptotic} and \eqref{eq:r_asymptotic} to functions counting the occurrences of any string $w$ of binary digits. More precisely, we let $\occ{w}{n}$ denote the number of occurrences of $w$ in the binary expansion of $n \in \N$. In particular, we have $\mathsf{s}(n) = \occ{\tl}{n}$ and $\mathsf{r}(n) = \occ{\tl \tl}{n}$. In the case when $w$ begins with a $\tO$ and contains a $\tl$, when computing $\occ{w}{n}$ we will use a standard convention (as in   \cite{AlloucheShallit2003}) that the binary expansion of $n$ is preceded by a block of leading zeros of suitable length  (see \eqref{eq:convention} below for a precise definition).

For each $t \in \N$ we define the function $d^w_t \colon \N \to \Z$, given  by
\begin{equation} \label{eq:d_def}
    d^w_t(n) = \occ{w}{n+t} - \occ{w}{n},
\end{equation}  
characterizing the change of $\occ{w}{n}$ under addition. 

\begin{remark}
% In the case when $w$ begins with a $\tO$ and contains a $\tl$, when computing $\occ{w}{n}$ we will use a standard convention that the binary expansion of $n$ is preceded by a block of leading zeros of suitable length  (see \eqref{eq:convention} for a precise definition). This improves the properties of $\occempty{w}$ and consequently $d^w_t$.
When $w = \tO^\ell$, that is, $w$ is a string of $\ell$ zeros, the behavior $d^{\tO^\ell}_t$ is slightly irregular when the binary expansions of $n, n+t$ have different lengths.  For the sake of convenience we will apply a ``correction'' to the definition of $d^{\tO^\ell}_t$ in this case (see \eqref{eq:d_modification} below). When $t$ is fixed, the set of such $n \in \N$ has density $0$, hence the modification does not affect our results.
% the above definition of $d^w_t(n)$ on a set of $n$ of density $0$, to accommodate for the slightly irregular behavior of $N^{\tO^\ell}$. More details are given in Section \ref{sec:basic}. This modification is done for the sake of convenience and does not affect our results. 
\end{remark}

As we will show in Proposition \ref{prop:arith_prog} below, for each $t \in \N$ and $k \in \Z$ the set
$$  \mathcal{D}^w_t(k) = \{n \in \N: d^w_t(n) = k\} $$
can be expressed as a (possibly infinite or empty) union of arithmetic progressions of the form $\{2^a m +b: m \in \N\}$, where $a \in \N$ and $b \in \{0,1,\ldots, 2^a-1 \}$. As a result, there exist natural densities
$$ \delta^w_t(k) = \dens \mathcal{D}^w_t(k),  $$
 which sum up to $1$ for each fixed $w, t$. Throughout the paper, we can thus identify $\delta^w_t$ with a probability distribution on $\Z$. 

We may exclude the case $w=\tO$ from our considerations, and assume that $w$ has length at least $2$. Indeed, when the binary expansions of $n$ and $n+t$ have the same length, we get $d_t^\tO (n)= - d_t^\tl (n)$. But for fixed $t$ the set of such $n$ has density $1$, and thus $\delta_t^\tO (j)= \delta_t^\tl (j)$ for all $j\in \Z$. Hence, the asymptotic formula \eqref{eq:s_asymptotic} holds when the left-hand side is replaced with $\delta_t^\tO(j)$.

 %Roughly speaking, if we randomly select $n \in \N$ by tossing a symmetric coin for each of its binary digits, then $\delta^w_t(k)$ is the probability that $d^w_t(n) = k$.

We now state our main result, where $v^w_t$ denotes the variance of the distribution $\delta^w_t$, equal to its second moment (as we will see, the mean is $0$ for each $t$):
$$v^w_t = \sum_{k \in \Z} k^2 \delta^w_t(k).$$
We note that $v^w_t$ can be computed using a set of recurrence relations, given in Corollary \ref{cor:v_formula}.

\begin{theorem} \label{thm:main}
Let $w$ be a string of binary digits of length at least $2$. For $t \in \N$ let $N$ denote the number of maximal blocks of $\tl$s in its binary expansion. Then for each $k \in \Z$ we have
\begin{equation}\label{eq:main}
 \delta^w_t(k)  = \frac{1}{\sqrt{2\pi v^w_t}}\exp\left(-\frac{k^2}{2v^w_t}\right) + O\left(\frac{(\log N)^2}{N}\right),
\end{equation}
as $N \to \infty$, where the implied constant depends only on $w$ and can be made explicit.
% \begin{equation}\label{eq:main}
% \left\lvert \delta_t(k)  - \frac{1}{\sqrt{2\pi v_t}}\exp\left(-\frac{k^2}{2v_t}\right)\right\rvert
% \leq C\frac{(\log \blocks{t})^2}{\blocks{t}}.
% %c_t(k) = \bigl(2\pi v_t\bigr)^{-1/2}\exp\biggl(-\frac{k^2}{2v_t}\biggr)+\LandauO%\bigl(
% %N^{-1}(\log N)^4\bigr),
% \end{equation}
\end{theorem}

\begin{remark}
As is the case with \eqref{eq:s_asymptotic} and \eqref{eq:r_asymptotic}, we can 
argue that over each interval $[-k_0,k_0]$ the main term dominates the error term asymptotically. We use Proposition \ref{prop:variance_ineq} below, which says that $ mN \leq v^w_t \leq M N$ for certain explicit constants $M > m > 0$. Choose a constant $C \in (0, \sqrt{m})$. If $N$ is large enough that $C \sqrt{N \log N} \geq k_0$, then for any $k \in [-k_0,k_0]$ we have
\begin{equation} \label{eq:small_error}
    \frac{1}{\sqrt{2\pi v^w_t}}\exp\left(-\frac{k^2}{2v^w_t}\right) \geq \frac{1}{\sqrt{2\pi M N}}\exp\left(-\frac{C^2 N \log N}{2mN}\right) \geq  \frac{1}{\sqrt{2\pi M}} N^{-(m+C^2)/2m},
\end{equation}   
where the exponent is strictly greater than $-1$ due to the choice of $C$. Hence, the last expression is asymptotically of higher order than $N^{-1} (\log N)^2$.
\end{remark}

The overall idea used to prove this result is similar as in \cite{SobolewskiSpiegelhofer2023} and relies on analyzing the behavior of the moments and characteristic functions $\gamma^w_t$ of the distributions $\delta^w_t$. The key technical ingredients are the following.
\begin{enumerate}
    \item[(A)] Linear bounds on the variance $v^w_t$ in terms of $N$.
    \item[(B)] Quality of approximation of $\gamma^w_t$ around $0$ by a Gaussian characteristic function.
    \item[(C)] An exponential upper bound on the tails of $\gamma^w_t$.
    \end{enumerate}
% Compared to the proof of Theorem \ref{thm:SW22} and \ref{thm:SS23}, these auxiliary results are considerably harder to prove due to $w$ being arbitrary.  
% \begin{enumerate}
%     \item upper and lower bounds on $v_t$;
%     \item estimating the quality Gaussian of approximation of $\gamma_{w,t}$;
%     \item bounding the tails of $\gamma_{w,t}$.
% \end{enumerate}
The main contribution of the present paper lies in showing that these properties still hold for arbitrary $w$. This presents much higher technical difficulties than in the previous results, as most of the time we cannot rely on formulas with concrete numerical coefficients, and instead need to resort to various approximations. Once properties (A)--(C) are established, the rest of the proof is essentially the same.

We now briefly outline the contents of the remainder of the paper. Section \ref{sec:notation} describes the notation and terminology. In Section \ref{sec:idea} we state precisely the three main technical ingredients (A)--(C) in our paper and show that they imply Theorem \ref{thm:main}. Later sections are devoted to proving these results. Basic properties and recurrences relations connecting the distributions $\delta^w_t$ for various $t$ are established in Section \ref{sec:basic}. To state these relations, we introduce certain conditional probability distributions which play an essential role throughout the whole proof. Along the way, we show that the densities $\delta^w_t(k)$ indeed exist. In Section \ref{sec:first} we study the means of these conditional distributions. Section \ref{sec:second} focuses their second moments and culminates in the proof of (A). Sections \ref{sec:normal} and \ref{sec:upper} contain the proof of (B) and (C), respectively. 
% Finally, Section \ref{sec:remarks} contains a short discussion concerning possible improvements and further research.

% \begin{theorem} \label{thm:main}
% Let $v_t$ denote the variance of the distribution $\delta_t$. There exist effective constants $C$, $T$ (depending only on $\ell$) such that for any  $t \in \N$ satisfying $|t|_{\tO\tl} \geq T$ we have
% \begin{equation}\label{eq:main}
% \left\lvert \delta_t(k)  - \frac{1}{\sqrt{2\pi v_t}}\exp\left(-\frac{k^2}{2v_t}\right)\right\rvert
% \leq C\frac{(\log |t|_{\tO\tl})^2}{|t|_{\tO\tl}}.
% \end{equation}
% % \begin{equation}\label{eq:main}
% % \left\lvert \delta_t(k)  - \frac{1}{\sqrt{2\pi v_t}}\exp\left(-\frac{k^2}{2v_t}\right)\right\rvert
% % \leq C\frac{(\log \blocks{t})^2}{\blocks{t}}.
% % %c_t(k) = \bigl(2\pi v_t\bigr)^{-1/2}\exp\biggl(-\frac{k^2}{2v_t}\biggr)+\LandauO%\bigl(
% % %N^{-1}(\log N)^4\bigr),
% % \end{equation}
% \end{theorem}

% However, despite great effort only partial results have been obtained so far. Drmota, Kauers and Spiegelhofer [??] showed that for any $\varepsilon >0$ the inequalities $1/2 < c_t < 1/2 +\varepsilon$ hold for almost all $t$ (in the sense of natural density).  

\section{Notation and terminology} \label{sec:notation}

We start with some notational and naming conventions related to binary words. Most of these are rather standard and follow \cite{AlloucheShallit2003}. When referring to binary digits, we always use typewriter font $\tO, \tl$.  
We let $\{\tO,\tl\}^*$ denote the set of all finite words (strings, blocks) on the alphabet $\{\tO,\tl\}$, including the empty word $\epsilon$. The length of $w \in \{\tO,\tl\}^*$ is denoted by $|w|$. The set of words of length $k \in \N$ is denoted by $\{\tO,\tl\}^k$. If $w=w_1w_2\cdots w_k$, where $w_j \in \{\tO,\tl\}$, then the reversal of $w$ is $w^R=w_k w_{k-1} \cdots w_1$. The bitwise negation of $w$ is $\overline{w} = \overline{w_1} \cdots \overline{w_k}$, where $\overline{\tO}=\tl, \overline{\tl}=\tO$.
The notation $\varepsilon^k$, where $\varepsilon \in \{\tO,\tl\}^*$, is a shorthand for $\varepsilon$ repeated $k$ times. 
For a word $w = w_m w_{m-1} \cdots w_1 w_0$, where $w_j \in \{\tO,\tl\}$, we let $[w]_2$ denote the integer represented by $w$ in base $2$, where the leftmost (leading) digit is the most significant, namely $[w]_2 = \sum_{j=0}^m 2^j w_j$. We note that leading zeros are allowed and will often appear when we require $w$ to have specified length. If $n \in \N$, then we refer to any $w$ satisfying $[w]_2=n$ as a binary expansion of $n$.  
Conversely, if $n \in \N$, then $(n)_2$ denotes the \emph{canonical} binary expansion $n$, namely the unique string $w \in \{\tO,\tl\}^*$ without leading zeros such that $[w]_2=n$ (in particular, $(0)_2 = \epsilon$). For $v,w \in \{\tO,\tl\}^*$ we say that $w$ is a prefix of $v$ if $v=wx$ for some $x \in \{\tO,\tl\}^*$. Similarly, $w$ is a suffix of $v$ if $v=xw$. More generally, $w$ is a subword (factor) of $v$ if $v=xwy$ for some $x,y\in \{\tO,\tl\}^*$. The number of occurrences of $w$ in $v$ as a subword is denoted by $|v|_w$, where we allow distinct occurrences to overlap. The function $\occempty{w}$ can be thus defined by 
\begin{equation} \label{eq:convention}
    \occ{w}{n} = | \tO^{|w|-1} (n)_2|_w,
\end{equation}
following the convention of preceding the canonical binary expansion with leading zeros. Note that this has no effect on the value of $\occ{w}{n}$ when $w$ starts with a $\tl$ or $w$ is a block of zeros. On the other hand, when $w$ starts with a $\tO$ (and contains a $1$) it sometimes causes an additional occurrence of $w$. For example, we have $(9)_2=\tl\tO\tO\tl$ but $\occ{\tO\tO\tl}{9} = |\tO\tO\tl\tO\tO\tl|_{\tO\tO\tl} = 2$. In particular, the number of maximal blocks of $\tl$s in the binary expansion of $t\in\N$, which often appears in our results, can be written as $\occ{\tO\tl}{t}$.

% By abuse of notation, for $n \in \N$ we will also write $|n|_u$ to denote the number of occurrences of $u$ in the binary expansion of $n$. Here we use the convention that if $u = \tO^j\tl x$  for some $j \in \N$ and $x \in \{\tO,\tl\}^*$, then $|n|_u= |\tO^jv|_u$, for any $v$ satisfying $[v]_2=n$.

We turn to the notation concerning matrices and vectors. Throughout the whole paper, their rows and columns will be indexed starting from $0$. We let $\mathbf{1}$ denote a column vector whose all entries are $1$. Similarly, $\mathbf{0}$ denotes a matrix of with all entries $0$. Their sizes will always be clear from the context.
For a matrix $A=[a_{jk}]_{0 \leq j \leq m, 0 \leq k \leq n}$ with complex entries we let $\|A\|_{\infty}$ denote its row-sum norm (infinity norm):
$$\|A\|_{\infty} = \max_{0 \leq j \leq m} \sum_{k=0}^n |a_{jk}|.$$
Finally, $i$ always denotes the imaginary unit and we write $\e(\theta)$ as a shorthand for $\exp(i \vartheta)$.

\section{Idea of the proof} \label{sec:idea}

Here and in the sequel we consider $w$ to be a fixed string of length $\ell = |w| \geq 2$, consisting of digits $\tO, \tl$. Therefore, we often omit $w$ in the superscript when it does not cause ambiguity, and simply write $N = N^w, d_t = d^w_t, \delta_t = \delta^w_t$, etc. 
% We point out that this convention does not apply to $|n|_w$ for $n \in \N$, to avoid confusion with the absolute value $|n|$.

The general idea of the proof of Theorem \ref{thm:main} is the same as in \cite{SpiegelhoferWallner2023, SobolewskiSpiegelhofer2023}. Nevertheless, we describe it in full for the sake of completeness. To begin, we state the three main technical ingredients from which Theorem \ref{thm:main} will follow. Firstly, we give linear bounds on the variance $v_t$. We note that the proof of Theorem \ref{thm:main} will only use the lower bound, however the upper bound was is needed in \eqref{eq:small_error} to prove that the error term in its statement is small. 

\begin{customprop}{A} \label{prop:variance_ineq}
    There exist constants $M > m > 0$ such that for all $t \in \N$ we have
    $$ m \occ{\tO\tl}{t} \leq  v^w_t \leq M \occ{\tO\tl}{t}, $$
    and we can choose
    $$ m = \frac{1}{4^{\ell-1}}, \qquad M =  \frac{3(\ell+2)}{2^{\ell-2}}.$$
\end{customprop}

Secondly, let $\gamma_t$ be the characteristic function of the distribution $\delta_t$:
$$  \gamma^w_t(\theta) = \sum_{k \in \Z} \delta_t(k) \e(k\theta). $$
We consider its Gaussian approximation $\hat{\gamma}^w_t$, defined by
$$  \hat{\gamma}^w_t(\theta) = \exp\left(-\frac{v_t}{2} \theta^2 \right). $$
The following proposition gives an estimate of the error of such approximation over an interval.
\begin{customprop}{B}    
 \label{prop:normal_approx}
    For any $\theta_0 > 0$ there exists a constant $K(\theta_0) > 0$ such that for all $t \in \N$ and $\theta \in [-\theta_0, \theta_0]$ we have    
    $$  |\gamma^w_t(\theta)-\hat{\gamma}^w_t(\theta)| \leq  K(\theta_0) \occ{\tO\tl}{t} |\theta|^3, $$
and we can choose
$$K(\theta_0) =  4 \ell \left(57 + \frac{19^2}{2}\theta_0  + \frac{(3+19\theta_0)^3}{6}\exp(3 \theta_0+19\theta_0^2 )\right).$$
    
%     where we can choose \textcolor{red}{rozwinac L}
% $$ L(a,b,\theta_0) =  |ab| + \frac{b^2}{2}\theta_0  + \frac{(|a|+|b|\theta_0)^3}{6}\exp(|a| \theta_0+|b|\theta_0^2 ).  $$
% % a=3 b=19
    
%     $$K(\theta_0) = 2(2\ell-1)L(3,19,\theta_0).$$    
\end{customprop}

Finally, we have an upper bound on $|\gamma_t(\theta)|$, given in terms of an exponential function.

\begin{customprop}{C}    
\label{prop:char_fun_bound}
    There exists a constant $L > 0$ such that for all $t \in \N$ satisfying $\occ{\tO\tl}{t} \geq \ell+3$ we have
    %where $x$ denotes any binary block of length $\ell-1$. then
    $$ |\gamma^w_t(\theta)| \leq  \exp(-L \occ{\tO\tl}{t} \theta^2), $$
    and we can choose
    $$ L = \frac{\pi^2}{2^{\ell+2}(\ell+3)}. $$
    %  $$ |\gamma_t(\theta)| \leq  \left(1- \frac{\theta^2}{2^{\ell+2}} \pi^2\right)^{K}, $$
    % where $K = \frac{2}{\ell+3}|t|_{01}-\frac{4}{3}$.
\end{customprop} 

As already mentioned, proving these results requires a substantial refinement of the approach used in \cite{SpiegelhoferWallner2023, SobolewskiSpiegelhofer2023}. In particular, we have made an effort to obtain possibly simple statements, while ensuring that the constants are explicit functions of $\ell$. With a more detailed analysis it should be possible to improve them even further.

We now show how these ingredients imply Theorem \ref{thm:main}. To simplify the notation, until the end of this section we will write $N = \occ{\tO\tl}{t}$. Moreover, we use symbolic constants, as in the above propositions, and their numerical values can be used to retrieve the implied constant in Theorem \ref{thm:main} as a function of $\ell = |w|$.
% Let $\gamma_t$ be the characteristic function of the distribution $\delta_t$:
% $$  \gamma_t(\theta) = \sum_{k \in \Z} \delta_t(k) \e(k\theta). $$

The value $\delta_t(k)$ can be extracted via the formula
\begin{equation} \label{eq:delta_int}
     \delta_t(k) = \frac{1}{2\pi} \int_{-\pi}^{\pi} \gamma_t(\theta) \e(-k \theta) \: d\theta.
\end{equation} 
We split this integral at $\pm \theta_0$, where 
$$  \theta_0 = C\sqrt{\frac{\log N}{N}},  $$
and $C > 0$  is a constant satisfying $C^2 \geq \max\{1/(2L), 1/m \}.$ Hence, we only consider $N$ sufficiently large, such that $\theta_0 \leq \pi$.

The integral over $[-\theta_0, \theta_0]$ is approximated by replacing $\gamma_t$ with $\hat{\gamma}_t$, and then extended to $\R$, which gives
$$  \frac{1}{2\pi}\int_{-\theta_0}^{\theta_0} \hat{\gamma}_t(\theta) \e(-k \theta) \: d\theta =   \frac{1}{2\pi}\int_{-\infty}^{\infty} \hat{\gamma}_t(\theta) \e(-k \theta) \: d\theta -  \frac{1}{\pi}\int_{\theta_0}^{\infty} \hat{\gamma}_t(\theta) \e(-k \theta) \: d\theta.$$
By the well-known inversion formula, the first integral equals the probability density function of normal distribution $\mathcal{N}(0,v_t)$, evaluated at $k$: 
\begin{equation} \label{eq:main_term}
    \frac{1}{2\pi}\int_{-\infty}^{\infty} \hat{\gamma}_t(\theta) \e(-k \theta) \: d\theta = \frac{1}{\sqrt{2 \pi v_t}} \exp\left( -\frac{k^2}{2v_t}\right).
\end{equation}
This is the main term in Theorem \ref{thm:main}. 

At the same time, for any $c>0$, we have the standard estimate
\begin{equation} \label{eq:tail_estimate}
    \int_{\theta_0}^{\infty} \exp(-c \theta^2) \: d\theta \leq \int_{\theta_0}^{\infty} \frac{\theta}{\theta_0} \exp(-c \theta^2) \: d\theta = \frac{1}{2c \theta_0 } \exp(-c \theta_0^2).
\end{equation} 
As a consequence, we get 
\begin{align*}
    \left|\int_{\theta_0}^{\infty} \hat{\gamma}_t(\theta) \e(-k \theta) \: d\theta \right| &\leq \int_{\theta_0}^{\infty} \exp\left(-\frac{v_t}{2} \theta^2 \right)  \: d\theta = \frac{1}{\theta_0 v_t} \exp\left(-\frac{v_t}{2} \theta_0^2 \right) \\
    &= \frac{1}{C v_t} \sqrt{\frac{N}{\log N}} N^{-C^2 v_t/(2N)} \leq \frac{1}{C m N \log N} = O\left(\frac{1}{N \log N}\right),
\end{align*}
where the last inequality follows from Proposition \ref{prop:variance_ineq} and $C^2 \geq 1/m$.

Now, by Proposition \ref{prop:normal_approx} the error introduced when approximating $\gamma_t$ is bounded by
$$ \left|\int_{-\theta_0}^{\theta_0} (\gamma_t(\theta) -\hat{\gamma}_t(\theta)) \e(-k \theta) \: d\theta \right| \leq  \int_{-\theta_0}^{\theta_0} K(\theta_0) N |\theta|^3 \: d\theta = O(N \theta_0^4) = O\left(\frac{ (\log N)^2}{N} \right). $$
Finally, the tails of the integral \eqref{eq:delta_int} can be estimated with the help of Proposition \ref{prop:char_fun_bound} and  \ref{prop:variance_ineq}:
$$
    \left|\int_{\theta_0\leq |\theta| \leq \pi} \gamma_t(\theta) \e(-k \theta) \: d\theta \right| \leq 2\int_{\theta_0}^{\pi} \exp(-LN\theta^2) \: d\theta \leq \frac{1}{L N \theta_0} \exp(-LN\theta_0^2) = O \left( \frac{1}{N \log N} \right),
$$  
where we have again used \eqref{eq:tail_estimate} and $C \geq 1/(2L)$. 

Adding up \eqref{eq:main_term} and the error terms, we get the statement of Theorem \ref{thm:main}.

\section{Basic properties and recurrences} \label{sec:basic}

In this section we establish some basic facts and derive recursive formulas, which will serve as a starting point towards proving Propositions \ref{prop:variance_ineq} -- \ref{prop:char_fun_bound}. 

\subsection{Existence of the densities}

To begin, we describe the modification of $d_t^w$ in the case $w=0^\ell$, mentioned in Section \ref{sec:intro}. More precisely, we define
\begin{equation} \label{eq:d_modification}
    d_{t}^{0^\ell}(n) = \occ{0^\ell}{n+t}- \occ{0^\ell}{n} - |(n+t)_2| + |(n)_2|,
\end{equation}
where the two last terms ``correct'' the behavior of the function.
In particular, if the canonical binary expansions $(n+t)_2$ and $(n)_2$ have identical length, then this formula is consistent with the definition \eqref{eq:d_def} for general $w$. Moreover, for each fixed $t$ the set of such $n$ has density $1$, hence the densities $\delta^{0^\ell}_t(k)$ remain the same regardless of this modification.
%Nevertheless, it is quite useful as it eliminates the need for distinguishing the case $w=0^\ell$ in the statements and proofs of several results below.

We now give a simple but useful fact, which says that $d_t(n)$ only depends on the digits affected by the addition $n+t$, together with a ``buffer'' of length $\ell-1$. The modification described above  allows us to state it in a consistent way for all $w$. 

\begin{lemma} \label{lem:no_carry}
Let $x,z \in \{\tO,\tl\}^*$ be such that $|x|=|z|$, and let $u \in \{\tO,\tl\}^{\ell-1}$. Then for any $v \in \{\tO,\tl\}^*$ we have
% If $w \neq 0^\ell$, then for any $v \in \{0,1\}^*$ we have 
$$  |vuz|_w - |vux|_w = |uz|_w-|ux|_w.  $$
% If $w=0^\ell$, then the same equality holds as long as $vu$ contains a $1$. 
In particular, if $t \in \N$ and $[z]_2=[x]_2+t$, then
    $$d_t([vux]_2) =  |uz|_w - |ux|_w.$$
% under the same assumptions as before.  
    % The following properties hold.
    % \begin{enumerate}
    %     \item[(a)] 
    %     \item[(b)]  Otherwise, if $u = 0^{\ell-1}$, then for all $v \in \{0,1\}^*$ containing a $1$  we have
    % $$ d_t([v0^{\ell-1}x]_2) = |10^{\ell-1}z|_w - |10^{\ell-1}x|_w.$$
    % \end{enumerate}        
\end{lemma}
\begin{proof}
Since $|u|=\ell-1$, any occurrence of $w$ in $vux$ (resp.\ $vuz$) must be either fully contained in either $vu$ or $ux$ (resp.\ $uz$). Hence,
$$  |vuz|_w - |vux|_w = |vu|_w +|uz|_w -( |vu|_w + |ux|_w ) = |uz|_w-|ux|_w. $$
In the second part of the statement, we have $[vux]_2 + t = 2^{|x|}[vu]_2+[x]_2+t = [vuz]_2$ (here we use the assumption $|x|=|z|$). If $w \neq \tO^\ell$, then
$$ d_t([vux]_2) = |0^{\ell-1}vuz|_w - |0^{\ell-1}vux|_w = |uz|_w-|ux|_w. $$  

If $w = \tO^\ell$ and $vu$ contains a $\tl$, then the canonical binary expansions of $n=[vux]_2$ and $n+t=[vuz]_2$ (obtained by removing leading zeros) have the same length, which again yields the desired formula. Finally, if  $w = \tO^\ell$ and $vu$ is a string of zeros, then we can write $z = \tO^k (n+t)_2$ and $x= \tO^m (n)_2$ for some $m \geq k \geq 0$. Then 
\begin{align*}
     d_t^{\tO^\ell}(n) &= |(n+t)_2|_{\tO^\ell} -|(n)_2|_{\tO^\ell} + |(n+t)_2| - |(n)_2| \\
     &= (|\tO^{\ell-1} z|_{\tO^\ell} - k) - (|\tO^{\ell-1} x|_{\tO^\ell} - m) -m+k = |uz|_{\tO^\ell}-|ux|_{\tO^\ell}.
\end{align*}   
    % \begin{align*}
    %      d_t(n) &= |0^{\ell-1}vuz|_w - |0^{\ell-1}vux|_w = (|0^{\ell-1}vu|_w+|uz|_w)-(|0^{\ell-1}vu|_w+|ux|_w) \\
    %      &= |uz|_w-|ux|_w = (|0^{\ell-1}u|_w+|uz|_w) -((|0^{\ell-1}u|_w+ |0^{\ell-1}ux|_w) \\
    %      &= |0^{\ell-1}uz|_w - |0^{\ell-1}ux|_w =  d_t(m). 
    % \end{align*}
    % If $w=0^{\ell}$ and $vu$ contains a $1$, the whole computation remains valid. 
\end{proof}

% \begin{remark}
%     In the case $w=0^\ell$ and $vu$ not containing a $1$ the lemma may not hold. For example, let $w=00$, $t=3$ and $x = 010, z=101$, so that $[z]_2 = [x]_2+t$. But for $u=0,v=0$, we get $d_t([vux]_2) = |101|_{00} - |10|_{00} = 0$, while $|uz|_{00} - |ux|_{00} = |0101|_{00} - |0010|_{00} = -1$.
% \end{remark}

We now argue that for each $t \in \N$ the densities $\delta_t(k)$ indeed exist. For a word $x \in \{\tO,\tl\}^*$ we use the notation 
$$\mathcal{A}(x) = 2^{|x|}\N + [x]_2,$$ that is, $\mathcal{A}(x)$ is the arithmetic progression consisting of $n \in \N$ whose binary expansion ends with $x$.
For $t \in \N$ put $h(t) = |(t)_2|$ and consider the family $\mathcal{F}_t$, consisting of all arithmetic progressions of the form $\mathcal{A}(ux)$, where $u \in \{\tO,\tl\}^{\ell-1}$ and $x$ satisfies one of the following conditions:
\begin{enumerate}
    \item [(I)] $x \in \{\tO,\tl\}^{h(t)}$ and $[x]_2 < 2^{h(t)}-t$;
    \item [(II)] $x=\tO\tl^s y$, where $s \in \N$ and $y \in \{\tO,\tl\}^{h(t)}$ is such that $[y]_2 \geq 2^{h(t)}-t$.
\end{enumerate}
Note that $\mathcal{F}_t$ is a partition of $\N$ and only depends on $w$ through $\ell$.

The following proposition shows that the sets $\mathcal{D}_{t}(k)$ are unions of progressions from $\mathcal{F}_t$ and gives some of their properties. 
\begin{proposition} \label{prop:arith_prog}
For each $t \in \N$ and $k \in \Z$ the set $\mathcal{D}_{t}(k)$ is a  (possibly infinite or empty) union of arithmetic progressions from $\mathcal{F}_t$. 
% If $w = 0^\ell$, then the same condition holds, up to a set $S_t \subset \N$ of density $0$.
% such that for each $k \in \Z$ the set $\mathcal{D}_{t}(k) \cap S_t$ is finite and $\mathcal{D}_{t}(k) \setminus S_t$ is a union of arithmetic progressions from $\mathcal{F}_t$.  

    Moreover, the following properties hold:
    \begin{enumerate}
        \item[(a)] if $w \not\in \{\tO^\ell,\tl^\ell\}$, then $\mathcal{D}_{t}(k) = \varnothing$ when $|k|> h(t)/2+3$;
        \item[(b)] if $w \in \{\tO^\ell,\tl^\ell\}$, then $\mathcal{D}_{t}(k)$ is a finite union of arithmetic progressions from $\mathcal{F}_t$, and their differences are at most  $2^{2h(t)+|k| +\ell-1}$.
    \end{enumerate}

    % If $w \neq 0^\ell$, then for each $t\in\N$ and $k \in\Z$ the set $\mathcal{D}_{t}(k)$ is a  (possibly infinite or empty) union of arithmetic progressions of the form $\{2^{a}m + b: m \in \N\}$, where $a \in \N$ and $b \in \{0,1,\ldots, 2^a-1\}$. 
    
    % In particular, we have the following properties:
    % \begin{enumerate}
    %     \item[(a)] if $w \not\in \{1^\ell,0^\ell\}$, then $\mathcal{D}_{t}(k) = \varnothing$ when $|k|> h/2+\ell$;
    %     \item[(b)] if $w = 1^\ell$, then $\mathcal{D}_{t}(k)$ is a finite union of arithmetic progressions of the above form, where $a \leq 2h+|k| +\ell-1$;
    %     \item[(c)] if $w = 0^\ell$, then there exists a set $S_t \subset \N$ of density $0$ such that for each $k \in \Z$ the intersection $\mathcal{D}_{t}(k) \cap S_t$ is finite, and  $\mathcal{D}_{t}(k) \setminus S_t$ is a finite union of sets of the form $S_t \cap (2^a \N + b)$ where $a \leq 2h+|k| +\ell-1$ and $0 \leq b < 2^a$.
    % \end{enumerate}
   \end{proposition}
\begin{proof}
% We first prove the result when $w \neq 0^\ell$, and later discuss the (slight) changes for $w = 0^\ell$.

% With each $x \in \{0,1\}^h$ we associate a family of arithmetic progressions, where we distinguish two cases.

Put $h = h(t)$ for brevity. Observe that for $\mathcal{A}(ux) \in \mathcal{F}_t$, if we write a binary expansion of $n \in \mathcal{A}(ux)$ as $vux$, the addition $n+t$ only affects the suffix $x$.
 More precisely, if $x$ is of the form (I), then let $z \in \{\tO,\tl\}^h$ be such that $[x]_2 + t = [z]_2$. By Lemma \ref{lem:no_carry}, for all $n \in \mathcal{A}(ux)$
 \begin{equation} \label{eq:d_caseI}
d_t(n) = |uz|_w - |ux|_w. 
\end{equation}
If $x$ is of the form (II), let $z \in \{\tO,\tl\}^h$ be such that $[y]_2 + t = [\tl z]_2$. Again, by Lemma \ref{lem:no_carry}, for all $n \in \mathcal{A}(ux)$ we have
\begin{equation} \label{eq:d_caseII}
     d_t(n) =   |u\tl\tO^sz|_w - |u\tO\tl^sy|_w.
\end{equation}
In either case, we can see that the function $d_t$ is constant on each progression in $\mathcal{F}_t$, which implies the main part of the statement.

% \textbf{Case I:} if $[x]_2+t < 2^h$, we consider progressions $\mathcal{A}(yx)$, where $y \in \{0,1\}^{\ell-1}$. Let $z \in \{0,1\}^h$ be such that $[x]_2 + t = [z]_2$. Then, by Lemma \ref{lem:no_carry}, for all $n \in \mathcal{A}(yx)$ we have
% \begin{equation} \label{eq:d_caseI}
% d_t(n) = |yz|_w - |yx|_w. 
% \end{equation}

% \textbf{Case II:} if $[x]_2+t \geq 2^h$, we consider progressions $\mathcal{A}(y01^sx)$, where  $y \in \{0,1\}^{\ell-1}$ and $s \geq 0$. Let $z \in \{0,1\}^h$ be such that $[x]_2 + t = [1z]_2$. Then, by Lemma \ref{lem:no_carry}, for all $n \in \mathcal{A}(y01^sx)$ we have
% \begin{equation} \label{eq:d_caseII}
%      d_t(n) =   |y10^sz|_w - |y01^sx|_w.
% \end{equation}

% If we write the elements of a given progression in the form $[vyx]_2$ or  $[vy01^sx]_2$ (depending on the case), where $v \in \{0,1\}^*$, then adding $t$ does not affect the prefix $vy$. Hence, by Lemma \ref{lem:no_carry}

% In either case, we can see that the function $d_t$ is constant on each of the considered progressions, that is, each of them is fully contained in one of the sets $\mathcal{D}_t(k)$. Moreover, the union of all these progressions is $\N$, which implies the first part of the statement. 

We now proceed to prove (a) and (b). Starting with (a), let $w \not\in \{\tO^\ell,\tl^\ell\}$. Consider an overlap of $w$, namely a word $r$ such that $w = w_1 r = r w_2$ for some nonempty $w_1,w_2$. Let $q = |r|$ be the maximal length of an overlap and put $p = \ell-q$. Then we have $0 \leq q \leq \ell-2$ (or $2 \leq p \leq \ell)$ and for any word $v \in \{\tO,\tl\}^*$ of length $|v| \geq \ell$ the inequality
$$ |v|_w \leq \frac{|v|-q}{p} = \frac{|v|-\ell}{p}  +1. $$
Applying this to \eqref{eq:d_caseI}, if $x$ is of the form (I) and $n\in\mathcal{A}(ux)$, we get
$$ |d_t(n)| \leq \frac{|ux|-\ell}{p}+1 \leq \frac{h -1}{2} +1 = \frac{h+1}{2}.$$
In the case (II), we need to bound the values $|u\tl\tO^sz|_w, |u\tO\tl^sy|_w$ appearing in \eqref{eq:d_caseII}. This is done similarly for both expressions so we focus on the first one. 
When $s \leq p$, we obtain
$$|u\tl\tO^sz|_w \leq \frac{|u\tl\tO^sz|-\ell}{p}+1 = \frac{s+h}{p}+1 \leq \frac{h}{2} + 2.$$
On the other hand, when $s > p$, the word $w$ cannot contain $\tO^s$ as a factor because of $w \neq \tO^\ell$. Hence, there are at most $2$ occurrences of $w$ in $u\tl\tO^sz$ which overlap with $0^s$: one which ends, and one which begins inside $\tO^s$. We thus obtain
$$|u\tl\tO^sz|_w \leq |u\tl|_w + 2 + |z|_w \leq 3 + \frac{h-\ell}{p}+1 \leq \frac{h}{2}+3. $$

Moving on to (b), we only consider $w=\tl^\ell$, as the proof for $w=\tO^\ell$ is similar. In the case (I) there are only finitely many progressions $\mathcal{A}(ux)$.
In the case (II) we have
$|u\tl\tO^sz|_w = |u\tl|_w+|z|_w \leq h-\ell+2$ and $|u\tO\tl^sy|_w \geq s-\ell+1$. This means that $ |d_t(n)| \geq \max\{ s-h+1,0\}$ for $n \in \mathcal{A}(ux)$. Thus, $\mathcal{D}_{t}(k)$ can be expressed as a finite union of arithmetic progressions: $\mathcal{A}(ux)$ of difference $2^{h+\ell-1}$, and $\mathcal{A}(u\tO\tl^sy)$ of difference $2^{h+\ell+s}$, where the inequality $|k| \geq s+h-1$ must be satisfied. After a small rearrangement, we get precisely (b). 

% Finally, we consider part (c), i.e., $w=0^\ell$. Let $S_t$ be the set consisting of the initial elements of the progressions $\mathcal{A}(0^{\ell-1} x)$ in the Case I, and $\mathcal{A}(0^{\ell-1}01^sx)$ in the Case II. In particular, we have the inclusion
% $$ S_t \subset \{ [1^sx]_2 : s \in \N, x \in \{0,1\}^h  \},  $$
% which shows that $\dens S_t = 0$. If $n \in \N \setminus S_t$, then by Lemma \ref{lem:no_carry} equality \eqref{eq:d_caseI} or \eqref{eq:d_caseII} holds (depending on the case), and the remaining part of (c)  follows from essentially the same argument as (b). 
\end{proof}

This result implies that the densities $\delta_t(k) = \dens \mathcal{D}_t(k)$ are indeed well-defined and sum to $1$ for fixed $t$. Therefore, $\delta_t$ can be viewed as a probability mass function on $\Z$.

In the following result, we show a symmetry relation between the distributions $\delta^w_t$ and $\delta^{\overline{w}}_t$. 
%The proof exhibits a correspondence between the sets $\mathcal{D}^w(k)$ and $\mathcal{D}^{\overline{w}}(-k)$.

\begin{proposition} \label{prop:symmetry}
    For all $t \in \N$ and $k \in \Z$ we have
    $$ \delta^{\overline{w}}_{t} (k) = \delta^{w}_{t} (-k).  $$
\end{proposition}
\begin{proof}
Consider the same family $\mathcal{F}_t$ of arithmetic progressions $\mathcal{A}(ux)$, as in the proof of Proposition \ref{prop:arith_prog}. 
% In the case $w = 0^\ell$ the reasoning below holds after excluding the elements of $S_t$, which has no effect on the densities. 
% We write them all in the common form $\mathcal{A}(yu)$, where $u=x$ in the Case I, or $u = 01^sx$ in the Case II. Moreover, in the case $w \in \{0^\ell,1^\ell\}$ the reasoning below holds up to elements of $S_t$ (as denoted in Proposition \ref{prop:arith_prog}), which has no effect on the densities. 
% We are going to exhibit a bijection $\sigma$ on $\mathcal{F}$, which preserves the difference of the progression and has the property $\sigma(\mathcal{A}$  (except for possibly one element, when $w \in \{0^\ell,1^\ell$\}). More precisely, for each $u$ we let $u' \in \{0,1\}^*$ be such that $|u'| = |u|$ and $[u]_2 + t = [u']_2$. We claim that the mapping $\mathcal{A}(yu) \mapsto \mathcal{A}(\overline{yu'})$ has the desired properties.
For each $x$ let $x' \in \{\tO,\tl\}^*$ be such that $|x'| = |x|$ and $[x]_2 + t = [x']_2$.
We claim that if  $\mathcal{A}(ux) \subset \mathcal{D}^w_t(k)$
for some $k \in \Z$, then  $\mathcal{A}(\overline{ux'}) \subset \mathcal{D}^{\overline{w}}_t(-k)$. 
Observe that $[\overline{x'}]_2 + t =  2^{|x|}-1 - [x']_2 + t = 2^{|x|}-1 - [x]_2 =  [\overline{x}]_2$.
Hence, by Lemma \ref{lem:no_carry} for all $n \in \mathcal{A}(ux)$ and $m \in \mathcal{A}(\overline{ux'})$ we have 
$$ d^w_t(n)  = |ux'|_w - |ux|_w = |\overline{ux'}|_{\overline{w}} - |\overline{ux}|_{\overline{w}} = -d^{\overline{w}}_t(m), $$
which proves our claim. 

To finish the proof, note that the mapping $\mathcal{A}(ux) \mapsto \mathcal{A}(\overline{ux'})$ is a bijection on $\mathcal{F}_t$ and preserves the density of the progression. Therefore, it also preserves density when extended to the sets $\mathcal{D}^w_t(k)$. 
\end{proof}

% Thanks to this result, we can see that \eqref{eq:r_asymptotic} implies Theorem \ref{thm:main} for $w=00$. Therefore, we can restrict our attention to $\ell = |w| \geq 3$.
% More generally, we may disregard the slightly anomalous case $0^\ell$ and will often do so to avoid unnecessary technicalities.

% \begin{remark}
% In a similar fashion it can be proved that $\delta^\tl_t(-k) = \delta^\tO_t(k)$. Therefore, \eqref{eq:s_asymptotic} implies an asymptotic formula of the same form for $\delta^\tO_t(k)$. 
% \end{remark}

% Our two main goals in this section are proving that the densities $\delta_t(k)$ are 

% these sets (and thus also $\mathcal{D}_t(k)$) have well-defined densities, and establishing a recurrence relation connecting characteristic functions of the arising probability distributions. More precisely $\mathcal{D}_{t,j}(k)$ are unions of certain arithmetic progressions (up to finitely many elements when $w=0^\ell$), and gives some of their properties. 

%Here and in the sequel $w \in \{0,1\}^*$ denotes a fixed binary word of length $\ell=|w| \geq 2$. 

\subsection{Recurrence relations} \label{subsec:recurrence}

Our next goal is to establish recurrence relations for the distributions $\delta_t$ and associated characteristic functions $\gamma_t$. 
To achieve this, we partition $\mathcal{D}_t(k)$ into $2^{\ell-1}$ subsets, each containing numbers with a different residue class modulo $2^{\ell-1}$. More precisely, for $j = 0,1, \ldots, 2^{\ell-1}-1$ and $k \in \Z$ we define
$$ \mathcal{D}_{t,j}(k) = \{n \in \N: d_t(2^{\ell-1}n+j)=k \} $$
so that the aforementioned partition is
$$ \mathcal{D}_t(k) = \bigcup_{j=0}^{2^{\ell-1}-1} \mathcal{D}_t(k) \cap (2^{\ell-1} \N + j) = \bigcup_{j=0}^{2^{\ell-1}-1} (2^{\ell-1} \mathcal{D}_{t,j}(k)+j).$$

We briefly discuss how Proposition \ref{prop:arith_prog} applies to the sets $\mathcal{D}_{t,j}(k)$. For fixed $k \in \Z$ write $\mathcal{D}_t(k) = \bigcup_{x \in X} \mathcal{A}(x)$, where $X \subset \{\tO,\tl\}^*$ is a finite set and $\mathcal{A}(x) \in \mathcal{F}_t$ for $x \in X$.
 % More precisely, let $\mathcal{A}(x) \in \mathcal{F}_t$ be a progression such that $\mathcal{A}(x) \subset \mathcal{D}_t(k)$. 
 Since each $\mathcal{A}(x)$ has difference $\geq 2^{\ell-1}$, the intersection $\mathcal{A}(x) \cap (2^{\ell-1} \N + j) = \mathcal{A}(x) \cap (2^{\ell-1} \mathcal{D}_{t,j}(k)+j)$ is either empty or equal to $\mathcal{A}(x)$. As a result, we have $\mathcal{D}_{t,j}(k) = \bigcup_{x \in X_j} \mathcal{A}(\widetilde{x})$ for some $X_j \subset X$, where $\widetilde{x}$ is obtained from $x$ by deleting its suffix of length $\ell-1$. By (a) in the same proposition, in the case $w \not\in \{1^\ell,0^\ell\}$ we have $\mathcal{D}_{t,j}(k) = \varnothing$ if $|k| > h/2 + \ell$. By (b), in the case $w \in \{\tO^\ell, \tl^\ell\}$ the union is finite, and the differences of $\mathcal{A}(\widetilde{x})$ can be bounded by $2^{2h+|k|}$, where $h=|(t)_2|$.

Moving on, it is clear that there exist densities
$$ \delta_{t,j}(k) = \dens \mathcal{D}_{t,j}(k), $$
and we have the equality
$$  \delta_t(k) = \frac{1}{2^{\ell-1}} \sum_{j=0}^{2^{\ell-1}-1} \delta_{t,j}(k). $$
If we introduce a probability distribution on $\N$, where the measurable sets are generated by $\mathcal{F}_t$ and their  probabilities are equal to densities, then $\delta_{t,j}$ equals the conditional probability of $\mathcal{D}_t(k)$, given $2^{\ell-1} \N + j$.

\begin{remark}
    The identity in Proposition \ref{prop:symmetry} can be refined to 
$$\delta^w_{t,j}(k) = \delta^{\overline{w}}_{t,j'}(-k),$$ where $j' = -(j+t+1) \bmod{2^{\ell-1}}$ (this is not needed for our main result). The proof is similar and left to the reader.
\end{remark}

We now derive recurrence relations involving the densities $\delta_{t,j}(k)$. To begin, unless $w=\tO^\ell, n=0$, the function $\occempty{w}$ satisfies the recurrence relation
$$ \occ{w}{n} = \occ{w}{\fl{n/2}} + 
\begin{cases}
1 &\text{if } n \equiv [w]_2 \pmod{2^{\ell}}, \\
0 &\text{if } n \not\equiv [w]_2 \pmod{2^{\ell}}.
\end{cases} $$
As a consequence, we obtain
\begin{equation} \label{eq:d_recurrence}
     d_t(n) = d_{t'} (\fl{n/2}) + \varphi(t,n),
\end{equation} 
where $t' = \fl{(t+n)/2} - \fl{n/2} =  \fl{t/2}+(tn \bmod{2})$, and
$$ \varphi(t,n) = \begin{cases}
1 &\text{if }  t \equiv [w]_2-n \not \equiv 0 \pmod{2^{\ell}}, \\
-1 &\text{if } t \not\equiv [w]_2 - n \equiv   0 \pmod{2^{\ell}}, \\
0 &\text{otherwise}.
\end{cases} $$
We note that \eqref{eq:d_recurrence} holds also for $w=\tO^\ell$ and all $n \in \N$ due to the convention \eqref{eq:d_modification}.
Using this relation, we obtain the following result.

\begin{proposition} \label{prop:set_rec}
    For all $t\in\N, j \in\{0,1,\ldots,2^{\ell-1}-1\}$ and $k \in \Z$ we have
    $$ \delta_{t,j}(k) = \frac{1}{2}\delta_{t',\fl{j/2}}(k- \varphi(t,j)) + \frac{1}{2}\delta_{t',\fl{j/2}+2^{\ell-2}}(k- \varphi(t,2^{\ell-1}+j)),$$
    where $t'= \fl{t/2}+(tj \bmod{2})$.
\end{proposition}
\begin{proof}
We have the equalities
\begin{align*} \mathcal{D}_{t,j}(k) &= \{2n \in \N: d_t(2^{\ell}n+j) = k\} \cup \{2n+1 \in \N: d_t(2^{\ell}n+2^{\ell-1}+j) = k\} \\
&= 2\{n \in \N: d_{t'}(2^{\ell-1}n+\fl{j/2}) + \varphi(t,j) = k   \} \\
&\cup (2\{n \in \N: d_{t'}(2^{\ell-1}n+2^{\ell-2}+\fl{j/2}) + \varphi(t,2^{\ell-1}+j) = k   \}+1) \\
&= 2\mathcal{D}_{t',\fl{j/2}}(k- \varphi(t,j)) 
\cup(2\mathcal{D}_{t',2^{\ell-2}+\fl{j/2}}(k- \varphi(t,2^{\ell-1}+j))+1),
\end{align*}
Corresponding recurrence relations for the densities $\delta_{t,j}(k)$ follow immediately.
\end{proof}

The initial conditions for $\delta_{t,j}$ are
$$\delta_{0,j}(k) =\begin{cases}
    1 &\text{if } k=0, \\
    0 &\text{if } k\neq 0.
\end{cases}$$
For $t=1$ and $j$ odd, the formula in Proposition \ref{prop:set_rec} also involves $\delta_{1,j'}$  on the right-hand side:
$$ \delta_{1,j}(k) = \frac{1}{2}\delta_{1,\fl{j/2}}(k- \varphi(1,j)) + \frac{1}{2}\delta_{1,\fl{j/2}+2^{\ell-2}}(k- \varphi(1,2^{\ell-1}+j)).   $$
This leads to a system of linear equations, where in the case $w\in\{\tO^\ell,\tl^\ell\}$ there are infinitely many variables, corresponding to nonzero densities $\delta_{1,j}(k)$. It is more convenient to recover these values using characteristic functions (see Proposition \ref{prop:char_fun_rec} below).

% It is not immediately obvious that the provided initial conditions allow the computation of $\delta_1$, and thus also $\delta_t$ for all $t \in \N$. Indeed, the final recurrence relation in Proposition \ref{prop:set_rec} applied to  reads

% As a corollary, for fixed $t$ we get bounds on the values of $d(t,n)$.

% \begin{corollary}
%     Let $t \in \N$ be fixed. Then for all $k \geq$ we have $\delta(t,k) =0$

%     Moreover, if 
% \end{corollary}

% \begin{remark}

%     From the proof of Proposition [] we can deduce some differences.
%     First, if $w = 1^{\ell}$, then

%     that for fixed $t \geq 1$ and $k \in \Z$ the set $\mathcal{D}(t,k,2^{\ell-1}-1)$ is a finite union of arithmetic progressions if and only if $w = 1^{\ell}$. Indeed, we have $$|u10^s|_{1^{\ell}}-|u01^s|_{1^{\ell}} = k$$ for only finitely many $s$. It follows that $\mathcal{D}(t,k)$ is also a finite union.

%     On the other hand, if $w \neq 1^{\ell}$, then
%     $$ |u10^s|_{1^{\ell}}-|u01^s|_{1^{\ell}} \in \{-1,0,1\},  $$
% and thus at least one of these possibilities is attained for infinitely many $s$. For example, if $w=u1$...
    
% \end{remark}

Let $\gamma_t$ denote the characteristic function of the distribution $\delta_t$, namely
$$  \gamma_t(\theta) =  \sum_{k \in \Z} \e(k \theta) \delta_t(k).   $$
Similarly the characteristic function of  $\delta_{t,j}$ will be denoted by
$$ \gamma_{t,j}(j,\theta) =  \sum_{k \in \Z} \e(k \theta) \delta_{t,j}(k).  $$
We then have the equality
$$ \gamma_t(\theta) = \frac{1}{2^{\ell-1}} \sum_{j=0}^{2^{\ell-1}-1}  \gamma_{t,j}(j,\theta). $$
Furthermore, Proposition \ref{prop:set_rec} translates into the recurrence relation
\begin{equation} \label{eq:gamma_rec}
    \gamma_{t,j}(\theta) = \frac{1}{2} \e(\varphi(t,j)\theta)  \gamma_{t',\fl{j/2}}(\theta) + \frac{1}{2} \e(\varphi(t,2^{\ell-1}+j) \theta) \gamma_{t',\fl{j/2}+2^{\ell-2}}(\theta),
\end{equation} 
where again $t'= \fl{t/2}+(tj \bmod{2})$.
We can rewrite it in a more convenient matrix form. First, define length $2^{\ell-1}$ column vectors
$$ \Gamma_t = \begin{bmatrix}
\gamma_{t,0} & \cdots & \gamma_{t,2^{\ell-1}-1}
\end{bmatrix}^T$$
and $2^{\ell-1} \times 2^{\ell-1}$ coefficient matrices
$$ A_t(\theta) = [a_t(j,k)]_{0 \leq j,k < 2^{\ell-1}}, \quad B_t(\theta) = [b_t(j,k)]_{0 \leq j,k < 2^{\ell-1}}, \quad C_t(\theta) = [c_t(j,k)]_{0 \leq j,k < 2^{\ell-1}}, $$
where 
\begin{align*}
a_{t}(j,\fl{j/2}) &= \frac{1}{2} \e(\varphi(t,j)\theta), \\
a_{t}(j,\fl{j/2}+2^{\ell-2}) &= \frac{1}{2} \e(\varphi(t,j+2^{\ell-1})\theta), \\
a_{t}(j,k) &= 0 \quad \text{otherwise},
\end{align*}
and
$$ b(j,k) = \begin{cases}
    a(j,k) &\text{if } j \text{ is even}, \\
    0 &\text{if } j \text{ is odd}, \\
\end{cases} $$
and
$$ c(j,k) = \begin{cases}
    0 &\text{if } j \text{ is even}, \\
    a(j,k) &\text{if } j \text{ is odd}. \\
\end{cases} $$
In other words, $B_t$ (resp.\ $C_t$) is obtained by replacing all entries in odd (resp.\ even) rows of $A_t$ with zeros (recall that indexing of entries starts at $0$). Visually, $A_t$ can be written the form
$$  A_t =   \frac{1}{2}
\begin{bmatrix}
    p_0 & 0 & \cdots & 0 & p_{2^{\ell-1}} & 0 & \cdots & 0 \\
    p_1 & 0 & \cdots & 0 & p_{2^{\ell-1}+1} & 0 & \cdots & 0 \\
    0 & p_2 & \cdots & 0 & 0 & p_{2^{\ell-1}+2} & \cdots & 0 \\
    0 & p_3 & \cdots & 0 & 0 & p_{2^{\ell-1}+3} & \cdots & 0 \\
    \vdots & \vdots &  \ddots & \vdots & \vdots & \vdots & \ddots & \vdots \\
    0 & 0 & \cdots & p_{2^{\ell-1}-2} & 0 & 0 & \cdots & p_{2^{\ell}-2} \\
    0 & 0 & \cdots & p_{2^{\ell-1}-1} & 0 & 0 & \cdots & p_{2^{\ell}-1}
\end{bmatrix},$$
where in the case $t \equiv 0 \pmod{2^\ell}$ we have $p_m = 1$ for all $m$, while for
$t \not\equiv 0 \pmod{2^\ell}$ we get
$$p_m = \begin{cases}
    \e(-\theta) &\text{if }   m = [w]_2, \\
    \e(\theta) &\text{if }   m \equiv [w]_2-t \pmod{2^{\ell}}, \\
    1  &\text{otherwise}.
\end{cases}$$
% Geometrically, if we wrap the matrix into a torus, then the nonzero entries form a ``double coil'' around it. 
% As $t$ increases, $e(-\theta)$ stays in one place, while $e(\theta)$ moves along the coil, visiting each position once. When $t \equiv 0 \pmod{2^{\ell}}$, these two entries cancel out to $1$.
% We can see that each of these matrices contains a double ``snake-like'' pattern, where the entry $e^{I\theta}$ moves up left along the ``snakes'' as $t$ increases and wraps around the edges.
% We can see that each of these matrices contains a double ``snake-like'' pattern, where the entry $e^{I\theta}$ moves up left along the ``snakes'' as $t$ increases and wraps around the edges.
For example, for $w=\tO\tl\tl$ we have:
\begin{align*}
A_0 &= \frac{1}{2}\begin{bmatrix}
1 & 0 & 1 & 0 \\ 1 & 0 & 1 & 0 \\ 0 & 1 & 0 & 1 \\ 0 & 1 & 0 & 1
\end{bmatrix},  \qquad
A_1 = \frac{1}{2}\begin{bmatrix}
1 & 0 & 1 & 0 \\ 1 & 0 & 1 & 0 \\ 0 & e^{I\theta} & 0 & 1 \\ 0 & e^{-I\theta} & 0 & 1
\end{bmatrix},  \qquad
A_2 = \frac{1}{2}\begin{bmatrix}
1 & 0 & 1 & 0 \\ e^{I\theta} & 0 & 1 & 0 \\ 0 & 1 & 0 & 1 \\ 0 & e^{-I\theta} & 0 & 1
\end{bmatrix}, \\
% A_3 = \frac{1}{2}\begin{bmatrix}
% e^{I\theta} & 0 & 1 & 0 \\ 1 & 0 & 1 & 0 \\ 0 & 1 & 0 & 1 \\ 0 & e^{-I\theta} & 0 & 1
% \end{bmatrix}, \\
% A_4 &= \frac{1}{2}\begin{bmatrix}
% 1 & 0 & 1 & 0 \\ 1 & 0 & 1 & 0 \\ 0 & 1 & 0 & 1 \\ 0 & e^{-I\theta} & 0 & e^{I\theta}
% \end{bmatrix}, 
% A_5 = \frac{1}{2}\begin{bmatrix}
% 1 & 0 & 1 & 0 \\ 1 & 0 & 1 & 0 \\ 0 & 1 & 0 & e^{I\theta} \\ 0 & e^{-I\theta} & 0 & 1
% \end{bmatrix}, 
% A_6 = \frac{1}{2}\begin{bmatrix}
% 1 & 0 & 1 & 0 \\ 1 & 0 & e^{I\theta} & 0 \\ 0 & 1 & 0 & 1 \\ 0 & e^{-I\theta} & 0 & 1
% \end{bmatrix},  
&\cdots, \qquad
A_6 = \frac{1}{2}\begin{bmatrix}
1 & 0 & 1 & 0 \\ 1 & 0 & e^{I\theta} & 0 \\ 0 & 1 & 0 & 1 \\ 0 & e^{-I\theta} & 0 & 1
\end{bmatrix},  \qquad
A_7 = \frac{1}{2}\begin{bmatrix}
1 & 0 & e^{I\theta} & 0 \\ 1 & 0 & 1 & 0 \\ 0 & 1 & 0 & 1 \\ 0 & e^{-I\theta} & 0 & 1
\end{bmatrix}, 
\end{align*}
and $A_{t+8} = A_t$.

We now state the main result of this section, which gives a recursive formula for the vectors $\Gamma_t$.

\begin{proposition} \label{prop:char_fun_rec}
For all $t \in \N$ we have
    \begin{align}
    \Gamma_{2t}(\theta) &=  A_{2t}(\theta) \Gamma_t(\theta ), \label{eq:Gamma_even} \\
    \Gamma_{2t+1}(\theta) &= B_{2t+1}(\theta) \Gamma_t(\theta) +  C_{2t+1}(\theta) \Gamma_{t+1}(\theta). \label{eq:Gamma_odd} 
\end{align}
Moreover, $\Gamma_0(\theta) = \mathbf{1}$ and $\Gamma_1(\theta)$ is the unique solution $\Gamma(\theta)$ of the system of equations
$$ (I - C_1(\theta)) \Gamma(\theta) = B_1(\theta) \mathbf{1},$$
where $I$ denotes the identity matrix of size $2^{\ell-1}$.
\end{proposition}
\begin{proof}
    The recurrence relations satisfied by $\Gamma_t$ follow straight from \eqref{eq:gamma_rec}. 
    
    It is also clear that all components of $\Gamma_0$ are identically equal to $1$, while the system of equations satisfied by $\Gamma_1$ is a consequence of identity \eqref{eq:Gamma_odd} applied to $t=0$. Hence, it remains to show that the determinant of the matrix $I - C_1(\theta)$ is not identically $0$. We will in fact prove a bit more, namely that
    $$  \det(I - C_1(\theta)) = \frac{1}{2} \begin{cases}
        2 - \e(\theta) &\text{if } w=0^{\ell}, \\
        2 - \e(-\theta) &\text{if } w=1^{\ell}, \\
        1  &\text{otherwise}.
    \end{cases} $$
    In other words, we claim that $\det(I - C_1(\theta))$ is equal to the product of the entries on the diagonal, where the only contribution of $-C_1(\theta)$ is $-1/2 \e(\varphi(1,2^\ell-1)\theta) $ in the bottom right entry.    
    To prove this, we show that swapping row (resp.\ column) $[x]_2$ with row (resp.\ column) $[x^R]_2$, for each $x \in \{0,1\}^{\ell-1}$ transforms $I - C_1$ into a lower-triangular matrix.
    % by moving the entry at position $(,[y]_2)$ to position $([x^R]_2,[y^R]_2)$, for each $x,y \in \{0,1\}^{\ell-1}$. First, note that such a mapping is indeed a permutation or rows and columns which fixes the diagonal. We only need to show that the image of $C_1$ under this mapping is lower-triangular. 
    Since $I$ is unaffected by these operations, it is sufficient to focus on $C_1$. By definition, this matrix has nonzero elements precisely at positions $([u1]_2,[0u]_2), ([u1]_2,[1u]_2)$, where $u \in \{0,1\}^{\ell-2}$. After performing the described operation we get a matrix with nonzero elements at positions $([1v]_2,[v0]_2), ([1v]_2,[v1]_2)$, where $v = u^R \in \{0,1\}^{\ell-2}$. It is easy to see that $[v0]_2 < [v1]_2 \leq [1v]_2$, and our claim follows.
\end{proof}

\begin{remark}
  One can as well write relations for ``glued'' vectors $[\Gamma_{2^{\ell-1}t}^T, \Gamma_{2^{\ell-1}t+1}^T, \cdots, \Gamma_{2^{\ell-1}(t+1)}^T ]^T$ in block form so that they only involve two distinct coefficient matrices (with blocks being $A_t,B_t,C_t$ and zero matrices), as was the case in \cite{SobolewskiSpiegelhofer2023}. In the present paper, we find it more convenient to deal with a larger number of smaller matrices.
\end{remark}

As mentioned earlier, having an expression for $\Gamma_t$, we can extract the densities $\delta_{t,j}(k)$. Moreover, by the proof of the proposition  (or Proposition \ref{prop:set_rec}(a)), $\gamma_{t,j}$ is an entire function of $\theta \in \C$ when $w \not\in \{\tO^\ell,\tl^\ell\}$. Even if $w \in \{\tO^\ell,\tl^\ell\}$, the poles of $\gamma_{t,j}$ can only be of the form $\theta = \pm i \log 2 + 2k\pi, k \in \Z$. Hence, we also get a corollary concerning the moments of corresponding probability distributions.

\begin{corollary}
    For each $t \in \N$ and $j = 0,1, \ldots,2^{\ell-1}-1$  all moments of $\delta_{t,j}$ exist and are finite.
\end{corollary}

In the remainder of this section we give some remarks and properties of the matrices $A_t, B_t, C_t$. First, since $\varphi(t,n)$ only depends on the residue of $t$ modulo $2^{\ell}$, it is clear that so do the matrices $A_t, B_t, C_t$. 
% Although only $A_{2t}$ and $B_{2t+1}, C_{2t+1}$ appear in the recurrence relations, it will be convenient to have these matrices defined for all indices $t$.
Moreover, we will often need to consider these matrices evaluated at $\theta=0$ so for brevity we introduce the notation
$$A = A_t(0), \qquad B = B_t(0), \qquad C=C_t(0),$$
where each expression is independent of $t$.

We now prove two important lemmas. The first one concerns the shape of the products of $B_t,C_t$, which naturally occurs when iterating the recurrence relation in Proposition \ref{prop:char_fun_rec}.

\begin{lemma} \label{lem:snake_matrix_product}
 Let $h \geq 1$ and choose any $\varepsilon_0, \varepsilon_1, \ldots, \varepsilon_{h-1} \in \{\tO,\tl\}$. Put $\uptau(\tO) = B$ and $\uptau(\tl) = C$. Then the entry at position $(j,k)$ in the matrix product $\uptau(\varepsilon_0) \uptau(\varepsilon_1) \cdots \uptau(\varepsilon_{h-1})$ is
 $$\begin{cases}
    1/2^h &\text{if } j=(2^{h}k + [\varepsilon_{h-1} \cdots\varepsilon_0]_2) \bmod{2^{\ell-1}}, \\
    0  &\text{otherwise}.
\end{cases}$$ 
 In particular, if $h \geq \ell-1$ then the product has all entries $1/2^h$ in row number $[\varepsilon_{\ell-2} \cdots \varepsilon_1\varepsilon_0]_2$, and $0$ elsewhere.

Consequently, for $h \geq \ell-1$ the matrix $ A^{h}$ has all entries equal to $1/2^{\ell-1}$.
\end{lemma}
\begin{proof}
Let $u_{j,k}^{(h)}$ denote the entry at position $(j,k)$ in the product  $\uptau(\varepsilon_0) \uptau(\varepsilon_1) \cdots \uptau(\varepsilon_{h-1})$. 
% We are going to prove a bit more, namely that
% $$u_{j,k}^{(h)} = \begin{cases}
%     1/2^h &\text{if } j=(2^{h}k + [\varepsilon_{h-1} \cdots\varepsilon_0]_2) \bmod{2^{\ell-1}}, \\
%     0  &\text{otherwise}.
% \end{cases}$$
The proof is by induction on $h$. For $h=1$ our claim is precisely the definition of $B$ and $C$. Now, assume that it holds for $h$ and consider multiplying  $\uptau(\varepsilon_0) \uptau(\varepsilon_1) \cdots \uptau(\varepsilon_{h-1})$ by $\uptau(\varepsilon_h)$. The matrix $\uptau(\varepsilon_h)$ has only one nonzero entry in column $k$, which lies in row $r=(2k+\varepsilon_h) \bmod{2^{\ell-1}}$ and equals $1/2$. Therefore, by the inductive assumption we get
$$  u_{j,k}^{(h+1)} = \frac{1}{2} u_{j,r}^{(h)} = \begin{cases}
    1/2^h &\text{if } j=(2^{h}(2k+\varepsilon_h) + [\varepsilon_{h-1} \cdots\varepsilon_0]_2) \bmod{2^{\ell-1}}, \\
    0  &\text{otherwise},
\end{cases}$$
which simplifies to the desired expression.

  The part concerning $A^h$ is obtained by expanding the product $A^h = (\uptau(0) +\uptau(1))^h$.
\end{proof}

The following lemma shows that for $k \geq 2\ell-2$ vectors $\Gamma_{2^k t}$ have all components equal to $\gamma_{2^kt}$.

\begin{lemma} \label{lem:char_fun_eq}
    For all $t \in \N$, $k \geq 2\ell-2$  and $j \in \{ 0,1, \ldots,2^{\ell-1}-1  \}$  we have $\gamma_{2^kt,j} = \gamma_{2^kt}. $
% \begin{enumerate}
% \item[(a)] $\gamma_{2^{2\ell-2}t,j}(\theta) = \gamma_{2^{2\ell-2}t}(j', \theta)$;
% \item[(b)] $\gamma_{2^{2\ell+k-2}t}(j,\theta) = \gamma_{2^{2\ell-2}t}(j,\theta)$.
% \end{enumerate}
\end{lemma}
\begin{proof}
By Proposition \ref{prop:char_fun_rec} we have $\Gamma_{2^{k}t} = A^{k-(\ell-1)}\Gamma_{2^{\ell-1}t}$. Since $k-(\ell-1) \geq \ell-1$, Lemma \ref{lem:snake_matrix_product} implies that all the components this vector are identical, and thus equal to their average $\gamma_{2^kt}$. 
\end{proof}

\section{First moments} \label{sec:first}

In this section we derive some important properties of the first moments $m_{t,j}$ of the distributions $\delta_{t,j}$, namely
$$ m_{t,j} = \sum_{k \in \Z} k \delta_{t,j}(k) = -i\gamma_{t,j}'(0).$$ We arrange them into column vectors
$$  M_t = \begin{bmatrix}
    m_{t,0} & m_{t,1} & \cdots & m_{t,2^{\ell-1}-1}
\end{bmatrix}^T. $$
By differentiating the relations for $\Gamma_t(\theta)$ in Proposition \ref{prop:char_fun_rec} with respect to $\theta$ and evaluating at $\theta=0$, we immediately get corresponding recurrence relations for $M_t$.

% \textcolor{red}{Is the next result obvious? Need to check the proof again}

\begin{proposition} \label{prop:mean_recurrence}
% We have $m(0,j) = 0$ for all $j=0,1,\ldots,2^{\ell-1}-1$ and
% $$ m(1,j) = \begin{cases}
%     1/2^{\ell-1}-1  &\text{if } j = 2^{\ell-1}-1, \\
%     1/2^{h+1} &\text{if } j=2^{\ell-1}-1-2^h \text{ for } 0 \leq h \leq \ell-2, \\
%     0 &\text{otherwise.}
% \end{cases}  $$
For all $t \in \N$ we have
    \begin{align*}
    M_{2t} &= A M_{t} + U_{2t},  \\
    M_{2t+1} &= B M_{t}+CM_{t+1} +U_{2t+1}, 
\end{align*}
where
$$ U_t = - i A'_{t}(0) \mathbf{1}. $$
\end{proposition}

% If we write $U_t = [u_t(0),u_t(1),\ldots,u_t({2^{\ell-1}-1})]^T$, then the $j$th component equals
% $$
% u_t(j) = \begin{cases}
%     \frac{1}{2} &\text{if } t \equiv [w]_2 - j \not\equiv 0 \pmod{2^{\ell-1}}, \\
%     -\frac{1}{2} &\text{if } t \not\equiv [w]_2 - j \equiv 0 \pmod{2^{\ell-1}}, \\
%     0 &\text{otherwise}.
% \end{cases}
% $$
% In particular, $U_t$ only depends on $t \bmod{2^{\ell-1}}$, where for $t \equiv 0 \pmod{2^{\ell-1}}$ it is the zero vector, and otherwise contains precisely two nonzero components: $1/2$ and $-1/2$.

Although these relations can be used to quickly compute the means numerically, for the purpose of proving our main result we will use other equivalent formulas. First, we provide an expression for $m_{t,j}$ in terms of a limit.

\begin{lemma} \label{lem:mean_limit}
    For all $t \in \N$ and $j=0,1,\ldots,2^{\ell-1}-1$ we have
    $$  m_{t,j} = \lim_{\lambda \to \infty}  \frac{1}{2^\lambda} \sum_{n=0}^{2^\lambda-1} d_t(2^{\ell-1}n+j). $$
%    \begin{align}
%     m_{t,j} &= \lim_{\lambda \to \infty}  \frac{1}{2^\lambda} \sum_{n=0}^{2^\lambda-1} d(t,2^{\ell-1}n+j) \nonumber
%     \\
%     &= \frac{1}{2^{\ell-1}}\sum_{u \in \{0,1\}^{\ell-1}} \left(|ux|_w - |uy|_w \right), \label{eq:mean_alternative}
% \end{align}
% where $x,y \in \{0,1\}^{\ell-1}$ are such that $[x]_2=(j+t) \bmod{2^{\ell-1}}$ and $[y]_2=j$.
\end{lemma}
\begin{proof}
We consider two cases, depending on $w$.
If $w \not\in\{\tO^\ell,\tl^\ell\}$, then by Proposition \ref{prop:set_rec}(a) only finitely many $\delta_{t,j}(k)$ are nonzero (for fixed $t$). It follows that
\begin{align*}
    m_{t,j} &= \sum_{k \in \Z} k \delta_{t,j}(k) = \sum_{k \in \Z} k \lim_{\lambda \to \infty} \frac{1}{2^{\lambda}} \#(\mathcal{D}_{t,j}(k) \cap [0,2^{\lambda})) \\
    &= \lim_{\lambda \to \infty}  \frac{1}{2^{\lambda}} \sum_{k \in \Z} \sum_{\substack{0 \leq n < 2^\lambda\\n\in \mathcal{D}_{t,j}(k)}} d_t(2^{\ell-1}n+j)  =\lim_{\lambda \to \infty} \frac{1}{2^{\lambda}} \sum_{n=0}^{2^\lambda-1} d_t(2^{\ell-1}n+j).
\end{align*}

Now, consider the case $w \in \{\tO^\ell,\tl^\ell\}$, where we need to be more delicate. Put $h = |(t)_2|$ and choose $\lambda \geq 4h$. By the discussion below the definition of $\mathcal{D}_{t,j}(k)$ in Subsection \ref{subsec:recurrence}, when $|k| \leq \lambda/2$, the set $\mathcal{D}_{t,j}(k)$ is a finite union of arithmetic progressions of the form $2^a \N + b$, where $a \leq 2h + |k| \leq \lambda$ and $0 \leq b < 2^a$.
% Put $h = |(t)_2|$ and choose $\lambda \geq 2h + \ell -1$. Define $k_{\lambda} = \lambda-(2h+\ell-1)$. By Proposition \ref{prop:set_rec}(b), if $|k| \leq k_{\lambda}$, then the set $\mathcal{D}_{t,j}(k)$ is (up to finitely many elements when $w=0^\ell$) a finite union of arithmetic progressions of difference at most $2^{\lambda}$ whose smallest element lies in $[0,2^\lambda)$.
% More precisely, fix $\lambda \geq \ell + 2L(t)$, where $L(t)$ denotes the length of the binary expansion of $t$. Using recurrence relations for the sets in the proof of the proposition, by induction on $L(t)$ the set $\mathcal{D}(t,j,k)$ is a union of arithmetic progressions of difference $2^{\lambda}$ for any $|k| \leq k_{\lambda} := \lambda-\ell-2L(t)+1$.
%Hence, we can take $k_{\lambda}$ to be the largest integer such that for all $|k| \leq  k_\lambda$ the set $\mathcal{D}(t,j,k)$ is a union of arithmetic progressions of difference $2^{\lambda}$.
Therefore, $\#\left(\mathcal{D}_{t,j}(k) \cap [0, 2^\lambda) \right)= 2^\lambda \delta_{t,j}(k)$, and we get
\begin{equation} \label{eq:truncated_mean}
     \sum_{|k| \leq \lambda/2} k \delta_{t,j}(k) = \frac{1}{2^{\lambda}} \sum_{|k| \leq \lambda/2} \sum_{\substack{0 \leq n < 2^\lambda\\n\in \mathcal{D}_{t,j}(k)}} d_t(2^{\ell-1}n+j) .
\end{equation}
% \begin{equation} \label{eq:truncated_mean}
%      \sum_{|k| \leq k_\lambda} k \delta_{t,j}(k) = \frac{1}{2^{\lambda}} \sum_{|k| \leq k_\lambda} \sum_{\substack{0 \leq n < 2^\lambda\\n\in \mathcal{D}_{t,j}(k)}} d_t(2^{\ell-1}n+j).
% \end{equation}
Let $R_\lambda$ denote an analogous sum over the remaining $n < 2^\lambda$, namely
% $$\hat{S}_{\lambda}  = \frac{1}{2^\lambda} \sum_{n=0}^{2^\lambda-1} d(t,2^{\ell-1}n+j)-S_\lambda 
$$ R_\lambda
= \frac{1}{2^\lambda} \sum_{|k|>\lambda/2} \sum_{\substack{0 \leq n < 2^\lambda\\n\in \mathcal{D}_{t,j}(k)}} d_t(2^{\ell-1}n+j) .$$
% where 
% $$K_{\lambda} = \max_{0 \leq n < 2^\lambda} |d_t(2^{\ell-1}n+j)|.$$
As $\lambda \to \infty$, the left-hand side of \eqref{eq:truncated_mean} tends to $m_{t,j}$, and it suffices to prove  $R_\lambda \to 0$. But for any $n < 2^\lambda$ we have $d_t(2^{\ell-1}n+j) \leq |(2^{\ell-1}n+j+t)_2|-\ell+1 \leq \lambda$, so $R_\lambda$ can be bounded by
$$  |R_\lambda| \leq \frac{\lambda}{2^\lambda} \sum_{|k| >\lambda/2} \#(\mathcal{D}_{t,j}(k)\cap[0,2^\lambda)) = \lambda \left(1-\sum_{|k| \leq \lambda/2} \delta_{t,j}(k) \right) = \lambda \sum_{|k| > \lambda/2} \delta_{t,j}(k) \leq  2 \sum_{|k| > \lambda/2} |k| \delta_{t,j}, $$
which tends to $0$ as $\lambda \to \infty$. The result follows.
% We have
% $$ |R_\lambda| \leq  \frac{K_\lambda}{2^\lambda} \sum_{k_\lambda < |k| \leq K_{\lambda} } \#(\mathcal{D}_{t,j}(k)\cap[0,2^\lambda)) \leq K_{\lambda}\sum_{k_\lambda < |k| \leq K_{\lambda} } \delta_{t,j}(k).  $$
% Since $|n|_w \leq \fl{\log_2 n}-\ell+2$  for any integer $n \geq 2^\ell-1$, we get
% $$ K_\lambda \leq \fl{\log_2(2^{\lambda+\ell-1} + j + t)}-\ell+2 = \lambda+1 \leq 2 k_\lambda, $$
% where the last inequality holds for all sufficiently large $\lambda$.
% Moreover, we know that
% $$ k_\lambda \sum_{k_\lambda < |k| \leq K_{\lambda}} \delta_{t,j}(k) \leq \sum_{|k| > k_\lambda} |k| \delta_{t,j}(k) \xrightarrow{\lambda \to \infty} 0,$$
% and the result follows.
% so it remains to find a suitable lower bound on $k_\lambda$ in terms of $\lambda$. The result follows from the fact that
% $$ \frac{\lambda}{k_{\lambda}} = \frac{\lambda}{\lambda-\ell-2L(t)+1} \leq 2$$
% for all sufficiently large $\lambda$.
% To this end, observe that $\mathcal{D}(t,k,j)$ can be expressed as a union of arithmetic progressions of difference $2^{\ell+|k|+2L(t)-1}$, where $L(t)$ denotes the length of the binary expansion of $t$. Indeed  This means that $k_\lambda \geq \lambda-\ell-2L(t)+1 \geq (q+2)/2$ for $q$ large enough. The result follows.
%, and consequently $k_q$ tends to $-\infty$ as $q$ grows to $\infty$.
\end{proof}

We now give two further equivalent expressions for $m_{t,j}$ in terms of finite sums, which will be useful for proving various properties of these means.
% One of them will involve the function $f \colon \{0,1\}^{\ell-1} \to \N$, defined by
% $$ f(v) := \sum_{u \in \{0,1\}^{\ell-1}} |uv|_w, $$
% so that 
% \begin{equation} \label{eq:mean_f}
%    m_{t,j} = \frac{1}{2^{\ell-1}}(f(x) - f(y))
% \end{equation}
% where $x,y$ are defined as in Lemma \ref{lem:mean_limit}. In order to evaluate $f(v)$ we will use the following simple observation.

\begin{lemma} \label{lem:mean_formula}
Let $t \in \N$ and $j \in \{0,1,\ldots,2^{\ell-1}\}$.
Let $x,y \in \{\tO,\tl\}^{\ell-1}$ be such that $j=[x]_2$ and $j+t \equiv [y]_2 \pmod{2^{\ell-1}}$. Then we have
\begin{equation} \label{eq:mean_formula1}
     m_{t,j} = \frac{1}{2^{\ell-1}}\sum_{u \in \{\tO,\tl\}^{\ell-1}}(|uy|_w-|ux|_w).
\end{equation}
Furthermore, for $v \in \{\tO,\tl\}^*$ let $\mathcal{P}(v)$ denote the set of nonempty prefixes of $v$ which are simultaneously suffixes of $w$. Then
\begin{equation}\label{eq:mean_formula2}
  m_{t,j} =  \frac{1}{2^{\ell-1}}\left(\sum_{p \in \mathcal{P}(y)} 2^{|p|-1}  - \sum_{p \in \mathcal{P}(x)} 2^{|p|-1} \right). \end{equation}

% For $n \in \N$ let $v_n \in \{0,1\}^{\ell-1}$ be such that $[v_n]_2 \equiv n \pmod{2^{\ell-1}}$.

% Moreover, for $v \in \{0,1\}^{\ell-1}$ let $\mathcal{P}(v)$ denote the set of nonempty prefixes of $v$ which are simultaneously suffixes of $w$. Then for all $t \in \N$ and $j \in \{0,1,\ldots,2^{\ell-1}-1\}$ we have 
% \begin{align}
%     m_{t,j} &= \frac{1}{2^{\ell-1}}\sum_{u \in \{0,1\}^{\ell-1}}(|uv_{j+t}|_w-|uv_j|_w) \label{eq:mean_formula1}\\
%     &= \frac{1}{2^{\ell-1}}\left(\sum_{p \in \mathcal{P}(v_{j+t})} 2^{|p|-1}  - \sum_{p \in \mathcal{P}(v_j)} 2^{|p|-1} \right).\label{eq:mean_formula2}
% \end{align}
    % Take any $v \in \{0,1\}^{\ell-1}$ and let $\mathcal{P}(v)$ denote the set of nonempty prefixes of $v$ which are simultaneously suffixes of $w$. Then
    % $$ f(v) = \sum_{p \in \mathcal{P}(v)} 2^{|p|-1}. $$
\end{lemma}
\begin{proof}
We first prove that \eqref{eq:mean_formula1} is equal to the limit in Lemma \ref{lem:mean_limit}. Let $k \in \N$ be such that $2^{\ell-1}k \leq j+t < (k+1)2^{\ell-1}$. If $w \neq \tO^\ell$, we have
\begin{align}
\sum_{n=0}^{2^\lambda-1} d_t(2^{\ell-1}n+j)  &= \sum_{n=0}^{2^\lambda-1} \left(N(2^{\ell-1}(n+k)+(j+t \bmod{2^{\ell-1}})) - N(2^{\ell-1}n+j) \right) \nonumber \\
&= \sum_{n=0}^{2^\lambda-1}\left(N(2^{\ell-1}n+(j+t \bmod{2^{\ell-1}})) - N(2^{\ell-1}n+j)\right) + E_\lambda, \label{eq:limit_approx}
\end{align}
where
$$ E_\lambda = \sum_{n=2^\lambda}^{2^\lambda+k-1}N(2^{\ell-1}n+(j+t \bmod{2^{\ell-1}})) - \sum_{n=0}^{k-1}N(2^{\ell-1}n+(j+t \bmod{2^{\ell-1}}))  = O(\lambda).   $$
If $w=\tO^\ell$, then due to \eqref{eq:d_modification} we also need to add to \eqref{eq:limit_approx} the sum
\begin{align*}
    \sum_{n=0}^{2^\lambda-1} \left( |(2^{\ell-1}n+j)_2|-|(2^{\ell-1}n+j+t)_2| \right)  &=  \sum_{n=0}^{2^\lambda-1} \left(|(n)_2| -|(n+k)_2|\right) \\
    &= \sum_{n=0}^{k-1} |(n)_2| -\sum_{n=2^\lambda}^{2^\lambda+k-1} |(n)_2|  = O(\lambda).
\end{align*}  
 Let $\lambda \geq \ell-1$ and write $n=[vu]_2$  for some $u \in \{0,1\}^{\ell-1}$ and $v \in \{\tO,\tl\}^{\lambda-\ell+1}$. Then the sum in \eqref{eq:limit_approx} becomes
$$ \sum_{v \in \{0,1\}^{\lambda-\ell+1}} \sum_{u \in \{\tO,\tl\}^{\ell-1}}(|vuy|_w-|vux|_w) = 2^{\lambda-\ell+1} \sum_{u \in \{\tO,\tl\}^{\ell-1}}(|uy|_w-|ux|_w).  $$
Hence, after dividing \eqref{eq:limit_approx} by $2^\lambda$ and passing to the limit, we get \eqref{eq:mean_formula1}.
 
Formula \eqref{eq:mean_formula2} follows from the identity
\begin{equation} \label{eq:double_counting}
    \sum_{u \in \{\tO,\tl\}^{\ell-1}}|uz|_w = \sum_{p \in \mathcal{P}(z)} 2^{|p|-1},
\end{equation}
valid for $z \in \{0,1\}^{\ell-1}$, which in turn is obtained by counting the occurrences of $w$ in two ways.
    Indeed, fix $z$ and observe that $w$ can only appear in $uz$ at positions such that the overlap of $w$ and $z$ equals some $p \in \mathcal{P}(z)$. For given $p$ this happens precisely when $u$ has suffix $s$ such that $sp=w$, and there are $2^{|p|-1}$ such choices of $u$. Summing over all $p \in \mathcal{P}(z)$, we get the total number of occurrences of $w$ in the words $uz$, which is the left-hand side of \eqref{eq:double_counting}.      
\end{proof}

As a simple consequence, we obtain a few key properties of the means $m_{t,j}$.

\begin{proposition} \label{prop:mean_properties}
For all $t \in \N$ we have the following:
  \begin{enumerate}
\item[(a)] $\sum_{j=0}^{2^{\ell-1}-1}m_{t,j} = 0$;
\item[(b)] $M_t$ is periodic in $t$ with period $2^{\ell-1}$;
\item[(c)] $\| M_t \|_{\infty} \leq 1-\frac{1}{2^{\ell-1}}$.
\end{enumerate}
\end{proposition} 
\begin{proof}
    To prove (a), observe that as $j$ runs over $0,1,\ldots,2^{\ell-1}-1$, then both $x$ and $y$ (as denoted in Lemma \ref{lem:mean_formula}) run over $\{\tO,\tl\}^{\ell-1}$. Therefore,
    $$ \sum_{j=0}^{2^{\ell-1}-1}m_{t,j} = \frac{1}{2^{\ell-1}} \left(\sum_{y \in \{\tO,\tl\}^{\ell-1}}\sum_{u \in \{\tO,\tl\}^{\ell-1}} |uy|_w -  \sum_{x \in \{\tO,\tl\}^{\ell-1}}\sum_{u \in \{\tO,\tl\}^{\ell-1}} |ux|_w \right) = 0.$$
The same lemma also shows that for fixed $j$ the value of $m_{t,j}$ only depends on $t \bmod{2^{\ell-1}}$, which gives (b). Finally, part (c) follows immediately from \eqref{eq:mean_formula2}.
\end{proof}

In particular, from (a) we get an immediate corollary concerning the original distribution $\delta_t$.
 \begin{corollary} \label{cor:mean_0}
     For every $t\in\N$ the mean of $\delta_t$ is $0$.
 \end{corollary}

\section{Second moments} \label{sec:second}

In this section we derive some important properties of the variance $v_t$ of the distribution $\delta_t$. Corollary \ref{cor:mean_0} implies that $v_t$ is equal to the second moment of $\delta_t$:
$$v_t = \sum_{k \in \Z} k^2 \delta_t(k) = -\gamma_t''(0).$$ 
 Our main goal is to prove Proposition \ref{prop:variance_ineq}, and for this purpose we establish a recurrence relation for $v_t$. This in turn requires us to consider the second moments $v_{t,j}$ of the distributions $\delta_{t,j}$, which are related to $v_t$ by the equality
$$  v_t = \frac{1}{2^{\ell-1}} \sum_{j=0}^{2^{\ell-1}-1} v_{t,j}.$$

The secondary goal of this section, fulfilled in Proposition \ref{prop:v_close} below, is be to prove that for all $j,k \in \{0,1, \ldots, 2^{\ell-1}-1\}$ the differences $v_{t,j} - v_{t,k}$ are uniformly bounded in $t$. This will turn out to be a key property in the approximation of $\gamma_t$ by the characteristic function of a normal distribution.

Let us define column vectors
$$ V_t = \begin{bmatrix}
    v_{t,0} & v_{t,1} & \cdots & v_{t,2^{\ell-1}-1}
\end{bmatrix}^T.  $$
By taking the second derivative of the relations in Proposition \ref{prop:char_fun_rec} at $\theta=0$ and multiplying by $-1$, we get the following recurrence relations.
\begin{proposition} \label{prop:var_recurrence}
For all $t \in \N$ we have
    \begin{align}
    V_{2t} &= A V_t + Q_{2t}, \label{eq:v_recurrence1} \\
    V_{2t+1} &= B V_t+C V_{t+1} + Q_{2t+1}, \label{eq:v_recurrence2}
\end{align}
where 
\begin{align*}
    Q_{2t} &= - 2i A'_{2t}(0) M_{t}-A''_{2t}(0) \mathbf{1}, \\
    Q_{2t+1} &= -2i  (B'_{2t+1} (0)M_{t}+C'_{2t+1}(0)M_{t+1})-A''_{2t+1}(0)\mathbf{1}.
\end{align*}
In particular the vectors $Q_t$ are periodic in $t$ with period $2^{\ell}$.
\end{proposition}

From this, we immediately obtain relations for the variances $v_t$. 

\begin{corollary} \label{cor:v_formula}
For all $t \in \N$ we have
\begin{align*}
    v_{2t} &= v_t + q_{2t}, \\
    v_{2t+1} &= \frac{1}{2}(v_t  + v_{t+1}) + q_{2t+1},
\end{align*}
where
$$   q_t = \frac{1}{2^{\ell-1}} \mathbf{1}^T Q_t.$$
\end{corollary}

The biggest obstacle towards proving Proposition \ref{prop:variance_ineq}, more precisely the $v_t \gg |t|_{01}$ part, is the fact that the values $q_t$ may not be strictly positive.  In the next lemma we show that, apart from some exceptional cases, we have $q_{t} \geq 1/2^{\ell+1}$, which will be sufficient for our purposes. The proof of this fact is rather technical and relies on a detailed case analysis. We also give an explicit upper bound for $q_t$.

\begin{lemma} \label{lem:q_bounds}
    For all $t \in \N$ we have
    $$ q_t < \frac{3}{2^{\ell-1}}.  $$
    Moreover, the inequality
    $$  q_{t} \geq \frac{1}{2^{\ell+1}} $$
    holds unless one of the following cases occurs:
    \begin{enumerate}[label=(\roman*)]
        \item $t \equiv 0 \pmod{2^{\ell}}$, where $q_t = 0$;
        \item $w \in \{\tO\tl^{\ell-1}, \tl\tO^{\ell-1}\}$ and $t \equiv 2^{\ell-1} \pmod{2^{\ell}},$ where
        $$ q_{t} = \frac{1}{2^{\ell-1}}\left(\frac{1}{2^{\ell-2}}-1\right); $$
        \item $w \in \{\tO\tl^{\ell-1}, \tl\tO^{\ell-1}\}$ and $t \equiv 2^{\ell-1} \pm 1 \pmod{2^{\ell}},$ where
        $$  q_t = \frac{1}{2^{2\ell-2}}; $$
        \item  $w \in \{\tO\tO\tl^{\ell-2},\tO\tl\tO^{\ell-2},\tl\tO\tl^{\ell-2},\tl\tl\tO^{\ell-2}\}$ and $t \equiv \pm 2^{\ell-2} \pmod{2^{\ell}}$, where
        $$  q_{t} \geq \frac{1}{2^{2\ell-2}}.$$
    \end{enumerate}    
\end{lemma}
\begin{proof}
By Proposition \ref{prop:symmetry}, the second moments corresponding to $w$ and its binary negation $\overline{w}$ satisfy $v^w_t = v^{\overline{w}}_t$ for each $t \in \N$, so we may assume without loss of generality that $w$ ends with a $\tO$. Furthermore, for $\ell=2$ we can verify that our claim holds by direct computation, so we only need to consider $\ell \geq 3$. In order to simplify the notation, throughout the proof we periodically extend the definition of $m_{t,j}$ to negative integers $t, j$, with period $2^{\ell-1}$ in both cases.

 To begin, we write $q_t$ more explicitly. If $t \equiv 0 \pmod{2^{\ell}}$, then $Q_t$ is the zero vector and we reach case (i). Hence, let $t \not \equiv 0 \pmod{2^{\ell}}$. Put $r = \fl{t/2}$ as well as $W = [w]_2$ (which is even) and $W' = W/2$. The matrix $A_t(\theta)$ has precisely two nonconstant entries: $\e(-\theta)/2$ at position $(W ,W') \bmod{2^{\ell-1}}$, and $\e(\theta)/2$ at position $(W-t,W'+\fl{-t/2}) \bmod{2^{\ell-1}}$. By the definition of $Q_t$ we thus obtain
$$
    2^{\ell-1} q_t = \mathbf{1}^T Q_t = 1+ \begin{cases}
        m_{r,W'-r} - m_{r,W'} &\text{if } t=2r, \\
        m_{r+1,W'-r-1}- m_{r,W'} &\text{if } t=2r+1.
    \end{cases}
$$
Note that in the case $t=2r+1$ we use the assumption that $w$ ends with a $0$ to deduce that $B_t(\theta)$ contains $\e(-\theta)/2$ and $C_t(\theta)$ contains $\e(\theta)/2$.
The upper bound for $q_t$ now follows from Proposition \ref{prop:mean_properties}(c).

% Moving on to the lower bound, by Lemma \ref{lem:mean_limit} for any $j$ we get
% $  m_r(j) = m_{-r}(j+r)$, which in conjunction with the above expressions yields  $\mathbf{1}^T Q_{t} = \mathbf{1}^T Q_{-t}$. Therefore, we can assume without loss of generality that $0 < r \leq 2^{\ell-2}$.     
%     Hence, let $t \not\equiv 0 \pmod{2^{\ell}}$. We have
%     $$ -\mathbf{1}_{2^{\ell-1}}^T A''_t(0) \mathbf{1}_{2^{\ell-1}} = 1,$$
%     since there are precisely two nonconstant entries in $A_t(\vartheta)$, equal to $\e(\theta)/2$ and $\e(-\theta)/2$.
        %         It remains to bound the first summand in the expression for $\mathbf{1}_{2^{\ell-1}}^T K_{t}$. Depending on the parity of $t$, we let $t = 2r$ or $t=2r+1$, and obtain
%     \begin{equation} \label{eq:K_summand_even}
%          - 2i \mathbf{1}_{2^{\ell-1}}^T A'_{2r}(0) M_{r} = m(r,\fl{[w]_2/2}-r) - m(r,\fl{[w]_2/2})
%     \end{equation}
%     or
%     \begin{equation} \label{eq:K_summand_odd}
%          - 2i \mathbf{1}_{2^{\ell-1}}^T (B'_{2r+1} (0)M_{r}+C'_{2r+1}(0)M_{r+1}) = m(r,\fl{[w]_2/2}-r) - m(r+1,\fl{[w]_2/2}),
%     \end{equation}
%     respectively. In both cases this expression is bounded from above  by $2- 1/2^{\ell-2}$ due to Lemma \ref{lem:mean_properties}(b), which yields the upper bound from the statement.
%     \textcolor{red}{jakaś notacja dla $[w]_2$ i $\fl{w/2}_2$?}
We move on to the inequality $q_t \geq 1/2^{\ell+1}$, where we need to estimate the means more precisely. We consider two main cases, depending on the parity of $t$.

\textbf{Case I:} $t=2r$. \\
If $2^{\ell-1} q_t \geq 1/4$, then we are done. Otherwise, we get
   \begin{equation} \label{eq:mean_ineq}
     m_{r,W'}  -  m_{r,W'-r} > \frac{3}{4}.
   \end{equation} 
To simplify the notation, for $x \in \{\tO,\tl\}^{\ell-1}$ we define 
$$ f(x) = \sum_{u \in \{\tO,\tl\}^{\ell-1}} |ux|_w = \sum_{p\in\mathcal{P}(x)} 2^{|p|-1}, $$
as in identity \eqref{eq:double_counting},
where again $\mathcal{P}(x)$ denotes the set of nonempty prefixes of $x$, which are suffixes of $w$. 
For $n \in \Z$ we also let $x_n \in \{\tO,\tl\}^{\ell-1}$ be such that $[x_n]_2 \equiv W'+n \pmod{2^{\ell-1}}$. By Lemma \ref{lem:mean_formula} we get 
   \begin{align}  
      2^{\ell-1}(m_{r,W'}-m_{r,W'-r}) &= f(x_r) + f(x_{-r}) - 2f(x_0) \nonumber \\  
        &\leq f(x_r) + f(x_{-r}) =  \sum_{p \in \mathcal{P}(x_r)} 2^{|p|-1} + \sum_{p \in \mathcal{P}(x_{-r})} 2^{|p|-1}. \label{eq:word_sum}
   \end{align}
   % \begin{align} 
   %     2^{\ell-1}(m_r(W'-r) - m_r(W')) &=2\sum_{u \in \{0,1\}^{\ell-1}} |ux_0|_w - \sum_{u \in \{0,1\}^{\ell-1}}\left(|ux_r|_w+|ux_{-r}|_w \right) \nonumber \\
   %     &= \sum_{s \in S(x_0)} 2^{|s|}-\left(\sum_{s \in S(x_r)} 2^{|s|-1}+\sum_{s \in S(x_{-r})} 2^{|s|-1}\right). \label{eq:word_sum}
   % \end{align}
    If \eqref{eq:mean_ineq} holds, then $f(x_r) + f(x_{-r}) > 3\cdot 2^{\ell-3}$ so at least one of the sums must contain $2^{\ell-2}$, or both must contain $2^{\ell-3}$. We consider these (not mutually exclusive) subcases separately.

    \textbf{Subcase Ia:} one of $f(x_r), f(x_{-r})$ contains $2^{\ell-2}$.\\    
    This condition holds if and only if at least one of $x_r, x_{-r}$ is the suffix of $w$ of length $\ell-1$. Since the sign of $r$ does not affect $f(x_r) + f(x_{-r})$, we can assume that this is happens for $x_r$. As a consequence, we get $r \equiv W' \pmod{2^{\ell-1}}$, and thus $x_{-r} = \tO^{\ell-1}$.
    %  or equivalently,  at least one of the following congruences is satisfied:
   % \begin{align*}
   %     W' + r &\equiv W \pmod{2^{\ell-1}}, \\
   %     W' - r &\equiv W \pmod{2^{\ell-1}}.       
   % \end{align*} 
   % Without loss of generality we can assume that the first congruence holds, as the sign of $r$ does not affect $f(x_r) + f(x_{-r})$. We get $r \equiv W' \pmod{2^{\ell-1}}$, which means that $x_r$ is obtained by deleting the first letter of $w$, while $x_{-r} = 0^{\ell-1}$.  
   For such $r$ we estimate \eqref{eq:word_sum} from above, depending on the number of zeros at the end of $w$. 
   
   If $w = \tO^{\ell}$, then $r \equiv 0 \pmod{2^{\ell-1}}$ so $t \equiv 0 \pmod{2^{\ell}}$, a case which has already been covered. If $w = \tl\tO^{\ell-1}$, then $x_r = x_{-r} = \tO^{\ell-1}$ and $x_0=\tl\tO^{\ell-2}$. Then it is simple to compute that $f(x_r) = f(x_{-r}) = 2^{\ell-1}-1$ and $f(x_0)=0$, which leads to  
   %the expression for $q_t$ given in the 
   case (ii). 
   Otherwise, write $w=\varepsilon y \tl\tO^s$, where $\varepsilon \in \{\tO,\tl\}$, $1 \leq s \leq \ell-2$ and $y \in\{\tO,\tl\}^{\ell-s-2}$ ($y$ may be the empty word). Then we get $x_r = y\tl\tO^s$ and $x_{-r}=\tO^{\ell-1}$, and therefore
   \begin{align*}
       \mathcal{P}(x_r) &= \{y\tl\tO^s\} \cup \mathcal{P}(y), \\
       \mathcal{P}(x_{-r}) &= \{\tO^h: 1 \leq h \leq s\}.
   \end{align*}
   We thus obtain
   $$f(x_{r}) + f(x_{-r}) \leq 2^{\ell-2} + \sum_{j=1}^{\ell-s-2} 2^{j-1} + \sum_{h=1}^s 2^{h-1} = 2^{\ell-2} + 2^{\ell-s-2} + 2^{s} - 2.$$
   Let us denote the expression of the right-hand side by $g(s)$. By elementary calculus, the maximal value of $g(s)$ for $s \in \{1,2,\ldots,\ell-2\}$ is attained when $s=\ell-2$, namely $w= \varepsilon \tl\tO^{\ell-2}$ for some $\varepsilon \in \{\tO,\tl\}$. We have $g(\ell-2)=2^{\ell-1}-1$, which leads to the inequality in the case (iv). To bound the same expression for other $w$ (when $\ell \geq 4$), we compute the second largest value of $g(s)$, which occurs for $s=\ell-3$ and $s=1$, and equals $3 \cdot 2^{\ell-3}$. But this is a contradiction with \eqref{eq:mean_ineq}.

\textbf{Subcase Ib:} both $f(x_{r}),f(x_{-r})$ contain $2^{\ell-3}$. \\
This condition holds if and only if the suffix of $w$ of length $\ell-2$ is a common prefix of $x_r$ and $x_{-r}$. This is equivalent to
$$  \fl{\frac{W'-r}{2}} \equiv  W  \equiv  \fl{\frac{W'+r}{2}} \pmod{2^{\ell-2}}.$$
From the congruence between the outer terms we get $r \equiv 0 \pmod{2^{\ell-2}}$, and further $r \equiv 2^{\ell-2} \pmod{2^{\ell-1}}$, since $2r = t \not \equiv 0 \pmod{2^{\ell}}$. Substituting this into the congruence, we obtain
$$  W + 2^{\ell-3} \equiv \fl{\frac{W}{4}} \pmod{2^{\ell-2}}. $$
Hence, $w$ must equal either $\tl\tO^{\ell-1}$ or one of the words  $\tO\tO(\tl\tO)^{\ell/2-1}, \tl(\tl\tO)^{(\ell-1)/2}$, depending on the parity of $\ell$. If $w = \tl\tO^{\ell-1}$, we get $r \equiv W' \pmod{2^{\ell-1}}$, which has already been considered in Subcase Ia. If $\ell$ is even and $w = \tO\tO(\tl\tO)^{\ell/2-1}$, then we get $x_r = x_{-r} = (\tl\tO)^{\ell/2-1}\tl$ and thus
$$ f(x_r) + f(x_{-r}) = 2 \sum_{h=1}^{\ell/2-1} 2^{2h-1} = \frac{2^{\ell}-4}{3}.$$
Similarly, if $\ell$ is odd and $w = \tl(\tl\tO)^{(\ell-1)/2}$, then $x_r = x_{-r} = (\tO\tl)^{(\ell-1)/2}$ and
$$ f(x_r) + f(x_{-r}) = 2 \sum_{h=1}^{(\ell-1)/2} 2^{2h-2} = \frac{2^{\ell}-2}{3}.$$
In both cases we have $f(x_r) + f(x_{-r}) < 3 \cdot 2^{\ell-3},$ thus a contradiction.

\textbf{Case II:} $t = 2r+1$. \\
The general reasoning is similar as in the previous case so we omit some details. Again, we assume that $2^{\ell-1} q_t < 1/4$, which is equivalent to
\begin{equation} \label{eq:f_inequality}
    3 \cdot 2^{\ell-3} <  2^{\ell-1}(m_{r,W'} - m_{r+1,W'-r-1}) = f(x_r)  +f(x_{-r-1})-2f(x_0)
\end{equation} 
and so we must have $f(x_r)  +f(x_{-r-1}) > 3 \cdot 2^{\ell-3}$.

  \textbf{Subcase IIa:} one of $f(x_r),  f(x_{-r-1})$ contains $2^{\ell-2}$.\\  
Again, one of $x_r, x_{-r-1}$ must be a suffix of $w$ and in either case the value of $f(x_r) + f(x_{-r-1})$ is the same. Hence, assume that $r \equiv W' \pmod{2^{\ell-1}}$, which means that $x_{-r-1} = \tl^{\ell-1}$, and thus $f(x_{-r-1}) = 0$ (recall that $w$ ends with a $\tO$).
% then $r \equiv W' \pmod{2^{\ell-1}}$ or $r+1 \equiv -W' \pmod{2^{\ell-1}}$. Again, either congruence yields the same value of $f(x_r) + f(x_{-r-1})$ so we can assume that $r \equiv W' \pmod{2^{\ell-1}}$.
% But then $x_r$ is obtained from $w$ by deleting the first letter, while .
Consider the number of zeros at the end of $w$.  
 %\textcolor{red}{moze wypisac $r$ zeby łatwiej odniesc sie do statementu?}
If $w = \tO^{\ell}$, then $x_r = x_0 =  \tO^{\ell-1}$, and the right-hand side of \eqref{eq:f_inequality} equals $1-2^{\ell-1}$, a contradiction. If $w=\tl\tO^{\ell-1}$, then  $x_0 = \tl\tO^{\ell-2}, x_r = \tO^{\ell-1}$, and the right-hand side of \eqref{eq:f_inequality} is $2^{\ell-1}-1$. This leads to $q_t = 1/2^{2\ell-2}$, where $t \equiv 2W'+1 \equiv 2^{\ell-1}+1 \pmod{2^\ell}$, which is case (iii).  Otherwise, $w$ ends with precisely $s$ zeros, where $1 \leq s \leq \ell-2$, and by similar calculations as before we have
$$ f(x_r) + f(x_{-r-1}) \leq \sum_{p \in \mathcal{P}(x_r)} 2^{|p|-1} \leq 2^{\ell-2} + 2^{\ell-3} - 1 < 3 \cdot 2^{\ell-3}, $$
a contradiction.

\textbf{Subcase IIb:} both $f(x_{r}),f(x_{-r-1})$ contain $2^{\ell-3}$. \\
The suffix of $w$ of length $\ell-2$ is a common prefix of $x_r$ and $x_{-r-1}$, and thus
$$  \fl{\frac{W'-r-1}{2}} \equiv  W  \equiv  \fl{\frac{W'+r}{2}} \pmod{2^{\ell-2}}.$$
Hence, $W'$ must be odd and either $r \equiv 0 \pmod{2^{\ell-2}}$ or $r \equiv -1 \pmod{2^{\ell-2}}$. We now consider the possible values of $r \bmod{2^{\ell-1}}$.

If $r \equiv 0 \pmod{2^{\ell-1}}$, then $w = (\tl\tO)^{\ell/2}$ or $w=\tO(\tl\tO)^{(\ell-1)/2}$, depending on the parity of $\ell$. Then $x_r = \tl(\tO\tl)^{\ell/2-1}$ or $x_r = (\tO\tl)^{(\ell-1)/2}$, respectively, and $x_{-r-1}$ differs from $x_r$ only in the final digit. Then the same computations as in the case $t=2r$ yield
$$ f(x_r) + f(x_{-r-1}) =  \frac{1}{3} \cdot\begin{cases}
   2^{\ell}-4&\text{if } \ell \text{ is even}, \\
    2^{\ell}-2 &\text{if } \ell \text{ is odd},
\end{cases}  $$
which is $< 3 \cdot 2^{\ell-3}$ both cases, thus a contradiction.

If $r \equiv 2^{\ell-2} \pmod{2^{\ell-1}}$, then only the first digit of $w$ is changed compared to the previous case, namely $w = \tO\tO(\tl\tO)^{\ell/2-1}$ or $w=\tl(\tl\tO)^{(\ell-1)/2}$. This yields precisely the same $x_r$ and $x_{-r-1}$, and thus the same value $f(x_r) + f(x_{-r-1})$. 

Finally, if $r \equiv -1 \pmod{2^{\ell-1}}$ or $r \equiv 2^{\ell-2}-1 \pmod{2^{\ell-1}}$, we obtain the same possibilities for $w$ and the set $\{x_r, x_{-r-1}\}$ as before, and thus the conclusion remains the same. 
\end{proof}

We state another auxiliary lemma which will be helpful in proving both main results of this section.

\begin{lemma} \label{lem:v_diff}
    For all $t \in \N$ we have
    $$ |v_{t+1} - v_t| \leq  \frac{1}{2^{\ell-4}}.$$
\end{lemma}
\begin{proof}
    By Corollary \ref{cor:v_formula} we have
\begin{align*}
    v_{2t+1} - v_{2t} &= \frac{1}{2}(v_{t+1}-v_t) + q_{2t+1}-q_{2t}, \\
    v_{2t+2} - v_{2t+1} &= \frac{1}{2}(v_{t+1}-v_t) + q_{2t+2}-q_{2t+1}.
\end{align*}
Since 
$$v_1 - v_0 = v_1 = 2q_1 = 2(q_1 - q_0),$$
by simple induction on $t$ and Lemma \ref{lem:q_bounds} we get
$$ |v_{t+1}-v_t| \leq 2\max_{t \in \N} |q_{t+1}-q_t| \leq 2\left(\frac{3}{2^{\ell-1}} + \frac{1}{2^{\ell-1}} \right)  = \frac{1}{2^{\ell-4}}. \qedhere$$
\end{proof}

We are now ready to prove Proposition \ref{prop:variance_ineq}.

% \begin{proposition} \label{prop:variance_ineq}
%     For all $t \in \N$ we have
%     $$ \frac{1}{2^{2\ell-2}} |t|_{01} \leq  v_t \leq \frac{3(\ell+2)}{2^{\ell-2}}|t|_{01}. $$
% \end{proposition}
\begin{proof}[Proof of Proposition \ref{prop:variance_ineq}]
Define
$$ r_t := \min \{ v_t, v_{t+1}\}, \qquad R_t := \max  \{ v_t, v_{t+1}\}. $$
 Our claim will follow from the chain of inequalities 
%\textcolor{red}{dopasowac do $s{t}$}
\begin{equation} \label{eq:v_inequalities}
    \frac{1}{2^{2\ell-2}} \occ{\tO\tl}{t} \leq r_t \leq v_t \leq R_t \leq  \frac{3(\ell+2)}{2^{\ell-2}}\occ{\tO\tl}{t},
\end{equation}
 for $t \geq 1$, where only the left- and rightmost are nontrivial. Concerning the lower bound on $v_t$, for the moment we ignore the special case $w \in \{\tO\tl,\tl\tO\}$ and treat it in a different way at the end of the proof.

For the sake of both inequalities, we establish relations between $v_t, v_{t+1}$ and $v_{t'}, v_{t'+1}$, where $t'$ is obtained by appending the block $\tO^k$ or $\tl^k$ to the binary expansion of $t$, for some $k \geq 1$.  By Corollary \ref{cor:v_formula}  we have the equalities
\begin{align}
    v_{2^kt}  - v_t &= \sum_{h=1}^k q_{2^h t} =: \Sigma_1, \label{eq:v_block0_even}\\
    v_{2^k t+1} - \left( \left(1-\frac{1}{2^k}\right)v_t + \frac{1}{2^k} v_{t+1} \right)&=  \sum_{h=1}^k \left(\left(1-\frac{1}{2^{k-h}}\right)q_{2^ht} + \frac{1}{2^{k-h}}q_{2^ht+1} \right) =: \Sigma_2, \label{eq:v_block0_odd}
\end{align}
where the second one follows from induction on $k$, and also
\begin{align}
    v_{2^kt+2^{k}-1}  -\left(  \frac{1}{2^k} v_t +  \left(1-\frac{1}{2^k}\right)v_{t+1} \right) &= \sum_{h=1}^k \left(\frac{1}{2^{k-h}}q_{2^ht+2^{h}-1} + \left(1-\frac{1}{2^{k-h}}\right)q_{2^h(t+1)} \right) =: \Sigma_3, \label{eq:v_block1_odd} \\
    v_{2^k (t+1)} -  v_{t+1}  &= \sum_{h=1}^k q_{2^h (t+1)} =:\Sigma_4. \label{eq:v_block1_even}
\end{align}
% We are going to show that for each $n=1,2,3,4$ we have
% \begin{equation} \label{eq:Sigma_inequalities}
%     \frac{1}{2^{2\ell-2}} \leq \Sigma_n \leq \frac{3 (\ell+1)}{2^{\ell-1}}.
% \end{equation} 
% Moreover, we assume that in each case the appended digits are different from the final binary digit of $t$, namely $t$ is odd in the former two equalities and it is odd in the two latter equalities.

Starting with the inequality for $R_t$, we need to bound all $\Sigma_n$ from above. For $h \geq \ell$ we have $q_{2^h t}=0$ by Lemma \ref{lem:q_bounds}(i). Applying the upper bound from the same lemma to the remaining summands in $\Sigma_1$, we get
$$  \Sigma_1 \leq \frac{3(\ell-1)}{2^{\ell-1}}, $$
and this inequality also holds for $\Sigma_4$. By the same properties we also get
$$ \Sigma_2 \leq \frac{3}{2^{\ell-1}} \left(\ell-1 + \sum_{h=\ell}^k \frac{1}{2^{k-h}} \right) < \frac{3(\ell+1)}{2^{\ell-1}}, $$
and again the same upper bound on $\Sigma_3$. Therefore, we obtain
$$  R_{2^k t} = \max\left\{ v_t + \Sigma_1, \left(1-\frac{1}{2^k}\right)v_t + \frac{1}{2^k} v_{t+1} + \Sigma_2     \right\} \leq R_t + \frac{3(\ell+1)}{2^{\ell-1}}, $$
and an identical bound on $R_{2^k t+2^k-1}$.
If we now divide the binary expansion $(t)_2$ into maximal blocks of $\tO$s and $\tl$s and let $\occ{*}{t}$ denote their number, then by induction we get
$$ R_t \leq R_0 + \frac{3(\ell+1)}{2^{\ell-1}} \occ{*}{t} \leq  R_0 + \frac{3(\ell+1)}{2^{\ell-2}} \occ{\tO\tl}{t}. $$
At the same time, Corollary \ref{cor:v_formula} and Lemma \ref{lem:q_bounds} imply
$$ R_0 = v_1 = 2q_1 \leq \frac{3}{2^{\ell-2}}, $$
and the upper bound on $R_t$ in \eqref{eq:v_inequalities} follows.

Moving on to the lower bound on $r_t$, we assume that $t$ is even in $\Sigma_1, \Sigma_2$, and odd in $\Sigma_3, \Sigma_4$.
% we assume that if $(t)_2$ ends with a $\tO$, then we append $\tl^k$, and vice versa. In other words, $t$ is even in $\Sigma_1, \Sigma_2$ and odd in $\Sigma_3, \Sigma_4$. 
Also, as mentioned before, for the moment we exclude the case $w =\tO\tl,\tl\tO$. If $w$ is not of the form appearing in cases (ii)--(iv) of Lemma \ref{lem:q_bounds} (which rules out $\ell=2$), then all summands corresponding to $h \leq \ell-1$ in each $\Sigma_n$ are at least $1/2^{\ell+1} \geq 1/2^{2\ell-2}$, and thus 
we get
\begin{equation} \label{eq:Sigma_inequality}
    \Sigma_n \geq \frac{1}{2^{2\ell-2}}.
\end{equation}
(In particular, the assumed parity of $t$ ensures that these summands are nonzero.)
The same inequality holds when $w \in \{\tO\tl\tO^{\ell-2},\tl\tl\tO^{\ell-2},\tO\tO\tl^{\ell-2},\tl\tO\tl^{\ell-2}\}$, where each summand is again bounded from below by $1/2^{2\ell-2}$.

% $v_{2^k t} = v_t$ and the sum in \eqref{eq:v_block0_odd} is nonnegative. This means that we only get
% $$  r_{2^k t} \geq r_t.  $$

If $w \in \{\tO\tl^{\ell-1},\tl\tO^{\ell-1}\}$ for some $\ell \geq 3$, then by  Lemma \ref{lem:q_bounds} for $t$ odd we have 
$q_{2^ht} \geq 1/2^{\ell+1}$ when $h <\ell-1$. At the same time, $q_{2^{\ell-1}t} > -1/2^{\ell-1}$, and $q_{2^ht} =0$ when $h > \ell-1$. In the case $\ell \geq 7$, this is enough to conclude that $\Sigma_1, \Sigma_4$ satisfy \eqref{eq:Sigma_inequality}. 
In the case $3 \leq \ell \leq 6$, the same inequality can be verified numerically for all odd $t < 2^{\ell}$, where it is sufficient to consider $k =\ell-1$ (since $q_{2^{\ell-1}t}$ is the only negative summand). 

Moving on to $\Sigma_2$, by Lemma \ref{lem:q_bounds}(ii) we have the following lower bounds on its summands:
$$ \left(1-\frac{1}{2^{k-h}}\right)q_{2^ht} + \frac{1}{2^{k-h}}q_{2^ht+1} \geq \begin{cases}
   1/2^{2\ell-2} &\text{if } h = 1, \\
   1/2^{\ell+1} &\text{if } 1 < h < \ell-1,\\
   -1/2^{\ell-1} &\text{if } h = \ell-1,\\
   0 &\text{if } h > \ell-1.
\end{cases}$$
Hence, if $k < \ell-1$, only two first cases can occur, and $\Sigma_2$ satisfies \eqref{eq:Sigma_inequality}. If $\ell \geq 7$ and $k \geq \ell-1$, we reach the same conclusion. Finally, for $3 \leq \ell \leq 6$ and $k \geq \ell-1$ we have the inequality
% $$  \Sigma_2 \geq \frac{1}{2^{2\ell-2}} + \frac{\ell-3}{2^{\ell+1}} -\frac{1}{2^{\ell-1}} \geq \frac{1}{2^{2\ell-2}} $$
$$\Sigma_2 = \Sigma_1 +  \frac{1}{2^k} \sum_{h=1}^{k} 2^h(q_{2^ht+1}-q_{2^ht}) \geq \frac{1}{2^{\ell+1}} + \frac{1}{2^k} \sum_{h=1}^{\ell-1} 2^h(q_{2^ht+1}-q_{2^ht}),$$
and it can be verified numerically that for any odd $t<2^\ell$ the sum on the right-hand side is positive. In a similar fashion we get $\Sigma_3 \geq 1/2^{\ell+1}$. This proves that \eqref{eq:Sigma_inequality} holds in all cases, except for possibly $w \in \{\tO\tl,\tl\tO\}$. This in turn implies the lower bound on $r_t$ in \eqref{eq:v_inequalities} by a similar induction as in the case of $R_t$.

To conclude, we deal with the lower bound on $v_t$ in the special case $w \in \{\tO\tl,\tl\tO\}$, where we need to modify the method a bit. Indeed, we have $q_0=q_2 =0$ and $q_1 = q_3 = 1/4$, which would only lead to $r_{2^k t} \geq r_t$ and $r_{2^kt+2^k-1} \geq r_t$ with no guaranteed increase. Hence, we focus on individual values $v_t$ rather than pairs $\{v_t,v_{t+1}\}$. By the proof of Lemma \ref{lem:v_diff} we get
$$ |v_{t+1}-v_t| \leq 2\max_{t \in \N} |q_{t+1}-q_t| = \frac{1}{2}.$$
% so for all $t \in \N$ the following relations hold:
% \begin{align*}
%     v_{2t} &= v_t, \\
%     v_{2t+1}&= \frac{1}{2}(v_t  +v_{t+1}) + \frac{1}{4}.
% \end{align*}
% It follows that
% \begin{align*}
%     v_{2t+1} - v_{2t} &= \frac{1}{2}(v_{t+1}-v_t) + \frac{1}{4}, \\
%     v_{2t+2} - v_{2t+1} &= \frac{1}{2}(v_{t+1}-v_t) - \frac{1}{4},
% \end{align*}
% and by simple induction we obtain $|v_{t+1} - v_t| \leq 1/2$. 
We thus obtain the inequalities
\begin{align*}
    v_{2t+1} - v_t &= \frac{1}{2}(v_{t+1} -v_t) +\frac{1}{4} \geq 0, \\
    v_{4t+1} - v_t &= \frac{1}{4}(v_{t+1} -v_t) +\frac{3}{8} \geq \frac{1}{4}.
\end{align*}
Combined with $v_{2t}=v_t$, this means that appending a digit to the binary expansion of $t$ does not decrease $v_t$, while appending the block $\tO\tl$ increases $v_t$ by at least $1/4$. Since $v_0=0$, we get
$$ v_t \geq \frac{1}{4} \occ{\tO\tl}{t}.  $$
We note that this method apparently does not work for general $w$, as then the bound on $v_{t+1}-v_t$ in Lemma \ref{lem:v_diff} is not tight enough to ensure nonnegativity of $v_{2t+1} - v_t$.
\end{proof}

\begin{remark}
   It should be possible to prove in a similar fashion that for all $k \geq 1$ the $k$th moment of the distribution $\delta_{t,j}$ is of order $\ll (\occ{\tO\tl}{t})^{k-1}$. We believe that for even $k$ the reverse inequality also holds, although the proof for a general $w$ would likely be very hard.
\end{remark}

To conclude this section, we give an upper bound on the differences $v_{t,j}-v_{t,k}$.

\begin{proposition} \label{prop:v_close}
    For all $t \in \N$ and $j,k \in \{0,1,\ldots,2^{\ell-1}-1\}$ we have
    $$ |v_{t,j} - v_{t,k}| \leq 16.$$
\end{proposition}
\begin{proof}
    Let us define the following block matrices:
    $$  \hat{D}_0 = \begin{bmatrix}
        A & \mathbf{0} \\ B & C
    \end{bmatrix}, \qquad  
    \hat{D}_1 =   \begin{bmatrix}
       B & C \\ \mathbf{0} & A
    \end{bmatrix}.$$
    We can rewrite Proposition \ref{prop:var_recurrence} as
    $$
       \begin{bmatrix}
           V_{2t} \\ V_{2t+1}
       \end{bmatrix} = \hat{D}_{0} \begin{bmatrix}
           V_t \\ V_{t+1}
       \end{bmatrix} +\begin{bmatrix}
           Q_{2t} \\ Q_{2t+1}
       \end{bmatrix}, \qquad
        \begin{bmatrix}
           V_{2t+1} \\ V_{2t+2}
       \end{bmatrix} = \hat{D}_1 \begin{bmatrix}
           V_t \\ V_{t+1}
       \end{bmatrix} +\begin{bmatrix}
           Q_{2t+1} \\ Q_{2t+2}
       \end{bmatrix}.
   $$
    Let $t \equiv [\varepsilon_{\ell-2} \cdots \varepsilon_1 \varepsilon_0]_2 \pmod{2^{\ell-1}}$, where $\varepsilon_h \in \{\tO,\tl\}$ for $h=0,1,\ldots,\ell-2$. After iterating the above relations $\ell-1$ times, we get
    $$  \begin{bmatrix}
           V_{t} \\ V_{t+1}
       \end{bmatrix} = \hat{D}_{\varepsilon_0} \hat{D}_{\varepsilon_1} \cdots  \hat{D}_{\varepsilon_{\ell-2}} \begin{bmatrix}
           V_{\fl{t/2^{\ell-1}}} \\ V_{\fl{t/2^{\ell-1}}+1}
       \end{bmatrix} + \sum_{h=0}^{\ell-2} \hat{D}_{\varepsilon_0} \hat{D}_{\varepsilon_1} \cdots  \hat{D}_{\varepsilon_{h-1}} \begin{bmatrix}
           Q_{\fl{t/2^h}} \\ Q_{\fl{t/2^h}+1}
       \end{bmatrix}.  $$
   For $h=0,1,\ldots,\ell-1$ we write in block form
    $$ \hat{D}_{\varepsilon_0} \hat{D}_{\varepsilon_1} \cdots  \hat{D}_{\varepsilon_{h-1}} = \begin{bmatrix}
        \hat{E}_h & \hat{F}_h \\ \hat{G}_h & \hat{H}_h
    \end{bmatrix},  $$
    where the blocks are of size $2^{\ell-1} \times 2^{\ell-1}$.  Then each of $\hat{E}_h, \hat{F}_h, \hat{G}_h, \hat{H}_h$ is either the zero matrix or a sum of distinct products of matrices $B,C$ of length $h$. Moreover, each such product contributes to precisely one of $\hat{E}_h,\hat{F}_h$ and precisely one of $\hat{G}_h,\hat{H}_h$. By Lemma \ref{lem:snake_matrix_product} this means that for each $j=0,1,\ldots,2^{\ell-1}-1$, the $j$th row of $\hat{E}_h$ (resp.\ $\hat{G}_h$) is nonzero if and only if the $j$th row of $\hat{F}_h$ (resp.\ $\hat{H}_h$) is zero. Furthermore, each nonzero row contains precisely $2^h$ nonzero entries, all equal to $1/2^h$.

        Combining this with the first part of Lemma \ref{lem:q_bounds}, which says that the sum of components of each of $Q_{\fl{t/2^h}}$ and $Q_{\fl{t/2^h}+1}$ is at most $3$, we obtain
    $$\left\|\sum_{h=0}^{\ell-2} \hat{D}_{\varepsilon_0} \hat{D}_{\varepsilon_1} \cdots  \hat{D}_{\varepsilon_{h-1}} \begin{bmatrix}
           Q_{\fl{t/2^h}} \\ Q_{\fl{t/2^h}+1}
       \end{bmatrix}\right\|_{\infty} \leq \sum_{h=0}^{\ell-2} \frac{3}{2^h} < 6.$$
    At the same time, by Lemma \ref{lem:snake_matrix_product} each component of the vector $\hat{E}_{\ell-1}V_{\fl{t/2^{\ell-1}}} + \hat{F}_{\ell-1}V_{\fl{t/2^{\ell-1}}+1}$ is equal to either $v_{\fl{t/2^{\ell-1}}}$ or $v_{\fl{t/2^{\ell-1}}+1}$. 
    % (depending on which rows of $E_{\ell-1}, F_{\ell-1},G_{\ell-1},H_{\ell-1}$ are nonzero).
% At the same time, each possible product of length $\ell-1$ of the matrices $B,C$ contributes to precisely one block in $D_{\varepsilon_0} D_{\varepsilon_1} \cdots D_{\varepsilon_{\ell-2}}$ (if we write $A=B+C$).  
% By Lemma \ref{lem:snake_matrix_product} this means that each row of this matrix is either $1/2^{\ell-1}\begin{bmatrix}
%     \mathbf{1}^T & \mathbf{0}^T
% \end{bmatrix}$ or $1/2^{\ell-1}\begin{bmatrix}
%     \mathbf{0}^T & \mathbf{1}^T
% \end{bmatrix}$
% Therefore,     
If that vector contains $v_{\fl{t/2^{\ell-1}}}$ at position $j$ and $v_{\fl{t/2^{\ell-1}}+1}$ at position $k$, we get
\begin{align*}
    |v_{t,j} - v_{t,k}| &\leq |v_{t,j}-v_{\fl{t/2^{\ell-1}}}| + | v_{t,k}-v_{\fl{t/2^{\ell-1}}+1}| + |v_{\fl{t/2^{\ell-1}}}-v_{\fl{t/2^{\ell-1}}+1}| \\
    &\leq 6 + 6 + \frac{1}{2^{\ell-4}} \leq 16,
\end{align*} 
where we have used Lemma \ref{lem:v_diff}. This inequality also holds if the roles of $j,k$ are reversed, or the $j$th and $k$th components are the same.
\end{proof}

\section{Gaussian approximation of the characteristic functions} \label{sec:normal}

In this section we study more closely the behavior of the characteristic functions $\gamma_{t}$. Our main goal is to prove Proposition \ref{prop:normal_approx}, that is, estimate the error of approximation by the characteristic function of normal distribution with the same variance $v_t$ and mean (equal to $0$). As in Section \ref{sec:idea}, the approximating function is defined by
$$\hat{\gamma}_t(\theta) = \exp \left(-\frac{v_t}{2}\theta^2\right),$$
and the error is
$$\widetilde{\gamma}_t(\theta) = \gamma_t(\theta) - \hat{\gamma}_t(\theta).$$
By definition, the function $\widetilde{\gamma}_t(\theta)$ expanded into a power series in $\theta$ only contains terms of order at least $3$, so we have $\widetilde{\gamma}_t(\theta) = O(\theta^3)$ as $\theta \to 0$. To investigate how the implied constant depends on $t$, we consider analogous approximations to the functions $\gamma_{t,j}$, and corresponding errors:
\begin{align*}
    \hat{\gamma}_{t,j}(\theta) &= \exp\left( i m_{t,j} \theta - \frac{u_{t,j}}{2} \theta^2 \right), \\
    \widetilde{\gamma}_{t,j}(\theta) &= \gamma_{t,j}(\theta) - \hat{\gamma}_{t,j}(\theta),
\end{align*} 
where $u_{t,j}$ is the variance of the distribution $\delta_{t,j}$:
$$  u_{t,j} = v_{t,j} - m_{t,j}^2. $$
A bound on $\widetilde{\gamma}_t$ will follow from an estimation of each $\widetilde{\gamma}_{t,j}$ as well as the difference
\begin{equation} \label{eq:gamma_difference}
    \widetilde{\gamma}_t - \frac{1}{2^{\ell-1}} \sum_{j=0}^{2^{\ell-1}-1} \widetilde{\gamma}_{t,j} = \hat{\gamma}_t -  \frac{1}{2^{\ell-1}} \sum_{j=0}^{2^{\ell-1}-1}  \hat{\gamma}_{t,j}.
\end{equation}
For later reference, we first state a rough upper bound on terms of order $\geq 3$ of certain exponential functions. Here and in the sequel, for a power series $g(\theta) = \sum_{n=0}^{\infty} c_n \theta^n$ we let $g_{\geq 3}$ denote the contribution of these terms:
$$g_{\geq 3}(\theta) = \sum_{n=3}^{\infty} c_n \theta^n.$$

\begin{lemma} \label{lem:order_3}
    Let $a,b, \theta_0 \in \R$, where $\theta_0 > 0$. There exists a constant $\mathcal{K}(a,b,\theta_0)$ such that for all $|\theta| \leq \theta_0$ we have
    $$ |\exp(ia\theta + b \theta^2)_{\geq 3}| \leq L(a,b,\theta_0) |\theta|^3,$$
    and we can put
   $$\mathcal{K}(a,b,\theta_0) =  |ab| + \frac{b^2}{2}\theta_0  + \frac{(|a|+|b|\theta_0)^3}{6}\exp(|a| \theta_0+|b|\theta_0^2 ).  $$
\end{lemma}
\begin{proof}
We have
$$ \exp(ia\theta + b \theta^2)_{\geq 3} = i ab \theta^3 + \frac{b^2}{2} \theta^4 + \sum_{n = 3}^{\infty} \frac{(ia \theta + b\theta^2)^n}{n!}$$
By Taylor's theorem, for any real $x \geq 0$ we get
$$   \sum_{n=3}^{\infty} \frac{x^n}{n!}\leq  \frac{x^3}{6} e^x,$$
and the statement follows by plugging in $x=|ia\theta + b\theta^2|$.
\end{proof}

We now arrange the functions $ \hat{\gamma}_{t,j}$ and errors $\widetilde{\gamma}_{t,j}$ into column vectors
$$   \hat{\Gamma}_t = \begin{bmatrix}
      \hat{\gamma}_{t,0} & \cdots &  \hat{\gamma}_{t,2^{\ell-1}-1}
\end{bmatrix}^T, \qquad  \widetilde{\Gamma}_t = \begin{bmatrix}
     \widetilde{\gamma}_{t,0} & \cdots & \widetilde{\gamma}_{t,2^{\ell-1}-1}
\end{bmatrix}^T.$$
The following lemma provides an upper bound on $\| \widetilde{\Gamma} (\theta)  \|_{\infty}$ over an interval around $0$, which is the main step in the proof of Proposition \ref{prop:normal_approx}.

\begin{lemma} \label{lem:normal_approximation}
    Let $t \in \N$ and $\theta_0 > 0$. Then there exists a constant $\mathcal{K}_1(\theta_0)$ such that for all $\theta \in [-\theta_0, \theta_0]$ we have
    $$  \|\widetilde{\Gamma}_t(\theta)\|_{\infty} \leq  \mathcal{K}_1(\theta_0) \occ{\tO\tl}{t} |\theta|^3, $$
    where we can put $\mathcal{K}_1(\theta_0) = (4 \ell-1) \mathcal{K}(3,19,\theta_0)$.
\end{lemma}
\begin{proof}
Let us define
$$  R_t(\theta) = \max\{ \|\widetilde{\Gamma}_t(\theta)\|_{\infty}, \|\widetilde{\Gamma}_{t+1}(\theta)\|_{\infty} \}. $$
For $t = 0$ our result holds because $\widetilde{\Gamma}_t(\theta)=0$, while for $t \geq 1$ we claim that the following, stronger inequality holds:
$$  R_t(\theta)  \leq \mathcal{K}_1(\theta_0) \occ{\tO\tl}{t}  |\theta|^3.$$
Let
\begin{align*}
    X_{2t} &= \widetilde{\Gamma}_{2t} - A_{2t} \widetilde{\Gamma}_t = A_{2t}\hat{\Gamma}_t - \hat{\Gamma}_{2t}, \\
    X_{2t+1} &= \widetilde{\Gamma}_{2t+1} - B_{2t+1} \widetilde{\Gamma}_t   - C_{2t+1}\widetilde{\Gamma}_{t+1} = B_{2t+1}\hat{\Gamma}_t + C_{2t+1}\hat{\Gamma}_{t+1} - \hat{\Gamma}_{2t+1}. 
\end{align*}
Roughly speaking, these are errors which arise when we apply to $\widetilde{\Gamma}_t$ the recurrence relation ``borrowed'' from $\Gamma_t$. 

We will iterate these relations written in block form, using the notation
$$  D_{2t} = \begin{bmatrix}
    A_{2t} & \mathbf{0} \\ B_{2t+1} & C_{2t+1}
\end{bmatrix}, \qquad D_{2t+1} = \begin{bmatrix}
    B_{2t+1} & C_{2t+1} \\ \mathbf{0} & A_{2t+2}
\end{bmatrix}. $$
In particular, we have $D_{2t}(0) = \hat{D}_0$ and $D_{2t+1}(0) = \hat{D}_1$, as denoted in the proof of Proposition \ref{prop:v_close}. 
By induction on $k \in \N$, we have
\begin{equation} \label{eq:Gamma_tilde_0}
    \begin{bmatrix}
    \widetilde{\Gamma}_{2^k t} \\
    \widetilde{\Gamma}_{2^k t+1}
\end{bmatrix} = D_{2^k t} D_{2^{k-1}t} \cdots D_{2t} \begin{bmatrix}
    \widetilde{\Gamma}_{t} \\
    \widetilde{\Gamma}_{t+1}
\end{bmatrix}  + \sum_{h=1}^{k} D_{2^k t} \cdots D_{2^{h+1}t} \begin{bmatrix}
    X_{2^h t} \\ 
    X_{2^h t+1}
\end{bmatrix},
\end{equation}  
as well as
\begin{align} 
\begin{bmatrix}
    \widetilde{\Gamma}_{2^k t + 2^k-1} \\
    \widetilde{\Gamma}_{2^k (t+1)}
\end{bmatrix} = &D_{2^kt+2^k-1} D_{2^{k-1}t+2^{k-1}-1} \cdots D_{2t+1} \begin{bmatrix}
    \widetilde{\Gamma}_{t} \\
    \widetilde{\Gamma}_{t+1}
\end{bmatrix} \nonumber \\
&+ \sum_{h=1}^{k} D_{2^kt+2^k-1} \cdots D_{2^{h+1}t+2^h-1} \begin{bmatrix}
    X_{2^h t+2^h-1} \\
    X_{2^h (t+1)}
\end{bmatrix}. \label{eq:Gamma_tilde_1} 
\end{align}
We need to bound the sums in both expressions, in particular the terms $X_t$.
%  $$   \widetilde{\Gamma}_{2^k t} = A_{2^k t} A_{2^{k-1}t} \cdots A_{2t} \widetilde{\Gamma}_{t}  + \sum_{h=1}^{k} A_{2^k t} \cdots A_{2^{h+1}t} X_{2^h t}.$$
%     By the previous observation, for any $\theta$ we get
%     $$ \|\widetilde{\Gamma}_{2^k t}(\theta)\|_{\infty} \leq \|\widetilde{\Gamma}_{t}(\theta)\|_{\infty} + \sum_{h=1}^{2\ell-2}  \|  X_{2^h t}(\theta)\|_{\infty}.$$
% Similarly, appending the block $1^k$ corresponds to
% $$  \begin{bmatrix}
%     \widetilde{\Gamma}_{2^k t + 2^k-1} \\
%     \widetilde{\Gamma}_{2^k (t+1)}
% \end{bmatrix} = D_{2^kt+2^k-1} D_{2^{k-1}t+2^{k-1}-1} \cdots D_{2t+1} \begin{bmatrix}
%     \widetilde{\Gamma}_{t} \\
%     \widetilde{\Gamma}_{t+1}
% \end{bmatrix} + \sum_{h=1}^{k} D_{2^kt+2^k-1} \cdots D_{2^{h+1}t+2^h-1} \begin{bmatrix}
%     X_{2^h t+2^h-1} \\
%     X_{2^h (t+1)}
% \end{bmatrix},$$
% while appending the block $1^k$ to
% We estimate $\| X_t(\theta) \|_{\infty}$ from above. 
By definition the $j$th component of $X_t(\theta)$ is of the form
\begin{equation} \label{eq:X_component}
    - \hat{\gamma}_{t,j}(\theta) + \frac{\e(\rho \theta)}{2} \hat{\gamma}_{t',r}(\theta) +\frac{\e(\sigma \theta)}{2} \hat{\gamma}_{t',s}(\theta) = \frac{\hat{\gamma}_{t,j}(\theta)}{2}\left(-2 + \e(\rho \theta)\frac{\hat{\gamma}_{t',r}}{\hat{\gamma}_{t,j}(\theta)} + \e(\sigma \theta)\frac{\hat{\gamma}_{t',s}(\theta)}{\hat{\gamma}_{t,j}(\theta)} \right),
\end{equation}   
where $t' \in \{\fl{t/2}, \fl{t/2}+1\}$ and $\rho, \sigma \in\{-1,0,1\}$, and $r,s \in \{0,1, \ldots, 2^{\ell-1}-1\}$. 
Since $\hat{\gamma}_{t, j}(\theta) = 1 + O(\theta)$ and the whole expression is $O(\theta^3)$, the expression in parentheses is also $O(\theta^3)$. Hence, it suffices to  bound terms of order $\geq 3$ of each summand there. We can write 
\begin{equation} \label{eq:first_summand}
    \e(\rho \theta)\frac{\hat{\gamma}_{t',r}(\theta)}{\hat{\gamma}_{t,j}(\theta)} = \exp(ia\theta+b\theta^2),
\end{equation}
where
$$  a = \rho + m_{t',r} - m_{t,j}, \qquad b = \frac{1}{2}(u_{t,j}-u_{t',r}).  $$
By Proposition \ref{prop:mean_properties} we get $|a| < 3$. At the same time,
\begin{align*}
    |2b| \leq |v_{t,j}-v_{t',r}| + |m_{t',r}^2-m_{t,j}^2 |  &\leq |v_{t,j}-v_{t}|+|v_{t}-v_{t'}|+|v_{t'}-v_{t',r}|+1- \frac{1}{2^{\ell-1}} \leq  \\
    &\leq 16 + \frac{3}{2^{\ell-1}} + \frac{1}{2^{\ell-4}} + 16 +1- \frac{1}{2^{\ell-1}} \leq 38,
\end{align*} 
 where the bounds on $v_{t,j}-v_{t}$ and $v_{t'}-v_{t',r}$ follow from Proposition \ref{prop:v_close}, while the bound on $v_{t}-v_{t'}$ is a consequence of Lemmas \ref{lem:q_bounds} and \ref{lem:v_diff}.
 Lemma \ref{lem:order_3} implies that the contribution of terms of order $\geq 3$ in \eqref{eq:first_summand} has absolute value bounded by $\mathcal{K}_2(\theta_0) |\theta|^3$, where we set 
 $$\mathcal{K}_2(\theta_0) = \frac{\mathcal{K}_1(\theta_0)}{4\ell-1} = \mathcal{K}(3,19,\theta_0) \geq \mathcal{K}(a,b,\theta_0).$$
 We can do exactly the same for the other non-constant summand $\e(\rho \theta)\hat{\gamma}_{t',r}(\theta)/\hat{\gamma}_{t,j}(\theta)$ in  \eqref{eq:X_component}.  
%  Let $\hat{K}(\theta_0)$ be the constant obtained by plugging the obtained bounds on $a,b$ into Lemma \ref{lem:order_3}, namely \textcolor{red}{moze pozniej dac tu $K$ razy cos}
%  $$  \hat{K}(\theta_0) = 51 + 17^2 \theta_0 + \frac{(3+17\theta_0)^3}{6}\exp(3\theta_0+17\theta_0^2). $$
%  We also define the constant $K(\theta_0)$ from the statement as
% $$K(\theta_0) = 2(2\ell-1)\hat{K}(\theta_0).$$
 Using $|\hat{\gamma}_{t,j}(\theta)| \leq 1$, we can bound \eqref{eq:X_component}, and therefore the whole vector $X_t$ in the following way:
\begin{equation} \label{eq:X_bound}
    \| X_{t}(\theta)\|_{\infty} \leq L |\theta|^3.
\end{equation}
Furthermore, Lemma \ref{lem:char_fun_eq} says that for $h \geq 2\ell-2$ the vector $\Gamma_{2^h t}$ has all components equal to $\gamma_{2^{2\ell-2}}$. Hence, $\hat{\Gamma}_{2^h t}$ has all components equal to $\hat{\gamma}_{2^{2\ell-2}}$, and we get
% Furthermore, Lemma \ref{lem:char_fun_eq} says that for any $t \in \N$ all components of $\Gamma_{2^{2\ell-2}t}$ are identical and equal to $\gamma_{2^{2\ell-2}t}$. Hence, all components of $\Gamma^*_{2^{2\ell-2}t}$ are equal to $\gamma^*_{2^{2\ell-2}t}$, and we obtain the following important property:
    \begin{equation} \label{eq:zero_vector}
        X_{2^{h}t} = \hat{\Gamma}_{2^{h}t} - A \hat{\Gamma}_{2^{h-1}t} = 0, \qquad h \geq 2\ell-1.
    \end{equation}  
    
We are now ready to bound the norm of \eqref{eq:Gamma_tilde_0} and \eqref{eq:Gamma_tilde_1}. Starting with the former expression,
by submultiplicativity we have 
$$   \left\| D_{2^k t}(\theta) D_{2^{k-1}t}(\theta) \cdots D_{2t}(\theta) \begin{bmatrix}
    \widetilde{\Gamma}_{t}(\theta) \\
    \widetilde{\Gamma}_{t+1}(\theta)
\end{bmatrix}  \right\|_{\infty} \leq R_t(\theta).  $$
Furthermore, for $h \leq \min\{2\ell-2,k\}$ by \eqref{eq:X_bound} we get
$$ \left\| D_{2^k t}(\theta) D_{2^{k-1}t}(\theta) \cdots D_{2^{h+1}t}(\theta) \begin{bmatrix}
   X_{2^ht}(\theta) \\
  X_{2^ht+1}(\theta)
\end{bmatrix}  \right\|_{\infty} \leq \mathcal{K}_2(\theta_0)|\theta|^3.  $$
In order to bound the terms with $h > 2\ell-2$ in \eqref{eq:Gamma_tilde_0} (when $k >  2\ell-2$), write in block form
$$ D_{2^k t} D_{2^{k-1}t} \cdots D_{2^{h+1}t} = \begin{bmatrix}
    E_h & \mathbf{0} \\ G_h & H_h
\end{bmatrix}. $$
In particular, we have $H_h(0) = C^{h-1}$ and by Lemma \ref{lem:snake_matrix_product} this matrix has entries $1/2^{h-1}$ in the bottom row and all other entries equal to $0$. Thus, by \eqref{eq:zero_vector} we get
$$ \left\| \begin{bmatrix}
    E_h & \mathbf{0} \\ G_h & H_h
\end{bmatrix}  \begin{bmatrix}
    X_{2^h t} \\ 
    X_{2^h t+1}
\end{bmatrix} \right\|_{\infty} \leq \left\|\begin{bmatrix}
    \mathbf{0} \\ H_h(\theta) X_{2^h t+1}(\theta)
\end{bmatrix} \right\|_{\infty} \leq  \frac{\mathcal{K}_2(\theta_0)}{2^{h-1}}|\theta|^3.$$
Combining all of the above, we obtain
$$ R_{2^kt}(\theta) \leq R_{t}(\theta) + (2\ell-2)\mathcal{K}_2(\theta_0) |\theta|^3 + \sum_{h=2\ell-1}^k \frac{\mathcal{K}_2(\theta_0)}{2^{h-1}} |\theta|^3 < R_{t}(\theta) + (2\ell-1)\mathcal{K}_2(\theta_0) |\theta|^3.$$

Moving on to \eqref{eq:Gamma_tilde_1}, in a similar fashion we can bound the expression before the sum as well as the terms corresponding to $h \leq 2\ell-2$. To bound the remaining terms, this time we put
$$ D_{2^k t + 2^k-1} \cdots D_{2^{h+1}t + 2^{h+1} -1} = \begin{bmatrix}
    E_h & F_h \\ \mathbf{0} & H_h
\end{bmatrix},$$
where, in particular, $E_h(0) = B^h$. Again, this  matrix has entries $1/2^{h-1}$ in the top row and all other entries equal to $0$, which combined with $X_{2^h(t+1)} = 0$ yields
$$ \left\| \begin{bmatrix}
    E_h & F_h \\ \mathbf{0} & H_h
\end{bmatrix}  \begin{bmatrix}
    X_{2^h t + 2^h -1} \\ 
    X_{2^h (t+1)}
\end{bmatrix} \right\|_{\infty} \leq \left\|\begin{bmatrix}
    E_h X_{2^h t + 2^h -1} \\ \mathbf{0}
\end{bmatrix} \right\|_{\infty} \leq  \frac{\mathcal{K}_2(\theta_0)}{2^{h-1}} |\theta|^3.$$
As a consequence, we again get 
$$ R_{2^kt+2^k-1}(\theta)  < R_{t}(\theta) + (2\ell-1)\mathcal{K}_2(\theta_0) |\theta|^3.$$
By induction on the number $\occ{*}{t}$ of maximal blocks of $\tO$s and $\tl$s in $(t)_2$, we thus obtain
$$  R_t(\theta) \leq R_0(\theta) + (2\ell-1)\mathcal{K}_2(\theta_0) \occ{*}{t} |\theta|^3 \leq R_0(\theta)+2(2\ell-1)\mathcal{K}_2(\theta_0) \occ{\tO\tl}{t} |\theta|^3.$$

It remains to show that $R_0(\theta) = \|\widetilde{\Gamma}_1(\theta)\|_{\infty} \leq \mathcal{K}_2(\theta_0) |\theta|^3$. By Proposition \ref{prop:char_fun_rec} we can explicitly write $\Gamma_1 = (I-C_1)^{-1} B_1 \mathbf{1}$. Without loss of generality we can assume that $w$ ends with a $\tO$, due to Proposition \ref{prop:symmetry}. Then the matrices $B_1$ and $C_1$ contain precisely one nonconstant entry: $\e(-\theta)/2$ and $\e(\theta)/2$, respectively. 
In the case $w \neq \tO^\ell$, since $\det(I - C_1(\theta)) = 1/2$, we can deduce that 
\begin{equation} \label{eq:gamma_1_form}
    \gamma_{1,j}(\theta) = \delta_{t,j}(-1) \e(-\theta) + \delta_{t,j}(0) + \delta_{t,j}(1) \e(\theta)
\end{equation} 
(another, more direct way would be to show that $|d_1(n)| \leq 1$).
% We separately bound the contribution of the terms of order $\geq 3$ in $\Gamma_1(\theta)$ and $\Gamma^*_1(\theta)$. 
% We use the description of $\Gamma_1$ from Corollary \ref{cor:Gamma_1}. In the first case of the corollary, the contribution of terms of order $\geq 3$ in $\gamma_{t,j}(\theta)$ can be bounded by
Therefore, by Lemma \ref{lem:order_3}
$$  |(\gamma_{t,j})_{\geq 3}(\theta)| \leq \delta_{t,j}(-1) \mathcal{K}(-1,0,\theta_0)  |\theta|^3+ \delta_{t,j}(1) L(1,0,\theta_0)  |\theta|^3 \leq \mathcal{K}(1,0,\theta_0) |\theta|^3.$$
At the same time,
$$ |(\hat{\gamma}_{t,j})_{\geq 3}(\theta)| \leq \mathcal{K}(m_{1,j},u_{1,j}/2,\theta_0) |\theta|^3 \leq \mathcal{K}(1,1/2,\theta_0) |\theta|^3. $$
It is easy to check that $\mathcal{K}(1,0,\theta_0) + \mathcal{K}(1,1/2,\theta_0) \leq \mathcal{K}_2(\theta_0)$, which proves our claim.

In the case $w = \tO^\ell$, we note that
$$ I-C_1 = \begin{bmatrix}
   P & \mathbf{0} \\ S & 1-\e(\theta)/2 
\end{bmatrix}, $$
where $P$ is a constant square matrix of size $2^{\ell-1}-1$. Hence, for $j\leq2^{\ell-1}-2$, the function $\gamma_{1,j}$ is again of the form \eqref{eq:gamma_1_form}. For $j=2^{\ell-1}-1$, writing out the last component in $\Gamma_1 = B_1 \mathbf{1} + C_1 \Gamma_1$ yields
 \begin{equation} \label{eq:gamma_1_form2}
     \gamma_{1,2^{\ell-1}-1} = \frac{\gamma_{1,2^{\ell-2}-1}(\theta)}{2 -\e(\theta)}. 
 \end{equation} 
Setting $j=2^{\ell-1}-1, j' = 2^{\ell-2}-1$ for brevity, we get
$$  \widetilde{\gamma}_{1,j}(\theta) = \frac{1}{2-\e(\theta)}  \left( \gamma_{1,j'}(\theta) - 2\exp \left(im_{1,j}\theta-\frac{u_{1,j}}{2} \theta^2\right) +  \exp \left(i(m_{1,j} + 1)\theta-\frac{u_{1,j}}{2} \theta^2\right) \right).$$
A direct calculation involving taking the second derivative of \eqref{eq:gamma_1_form2} yields a rough bound $u_{1,j} \leq v_{1,j} \leq 6$.
Bounding the absolute value of the denominator from below by $1$ and again applying Lemma \ref{lem:order_3} to each term in the parentheses, we reach the inequality 
$$|\widetilde{\gamma}_{1,j}(\theta)| \leq  (\mathcal{K}(1,0,\theta_0)+2\mathcal{K}(1,3,\theta_0) + \mathcal{K}(2,3,\theta_0))|\theta|^3.$$
A straightforward calculation shows that the constant is still $<\mathcal{K}_2(\theta_0)$, and the result follows.
\end{proof}

We are now ready to finish the proof of Proposition \ref{prop:normal_approx}.

% \begin{proposition} \label{prop:normal_approx}
%     Let $t \in \N$ and $\theta_0 > 0$. Then there exists a constant $K(\theta_0)$ such that for all $\theta \in [-\theta_0, \theta_0]$ we have    
%     $$  |\widetilde{\gamma}_t(\theta)| \leq  K(\theta_0)(|t|_{01} + 2) |\theta|^3, $$
%     where we can put $K(\theta_0) = 2(2\ell-1)L(3,19,\theta_0)$.    
% \end{proposition}
\begin{proof}[Proof of Proposition \ref{prop:normal_approx}]
The result holds for $t=0$ so we can assume that $t \geq 1$.
    Using identity \eqref{eq:gamma_difference} and Lemma \ref{lem:normal_approximation}, we get
    \begin{align*} 
    |\widetilde{\gamma}_t(\theta)| &\leq \frac{1}{2^{\ell-1}} \sum_{j=0}^{2^{\ell-1}-1} |\widetilde{\gamma}_{t,j}(\theta)| + \frac{1}{2^{\ell-1}}|\hat{\gamma}_{t,j}(\theta)| \left|\sum_{j=0}^{2^{\ell-1}-1} \frac{\hat{\gamma}_{t,j}(\theta)}{\hat{\gamma}_t(\theta)}\right| \\
    &\leq \mathcal{K}_1(\theta_0) \occ{\tO\tl}{t} |\theta|^3 + \frac{1}{2^{\ell-1}} \left|\sum_{j=0}^{2^{\ell-1}-1} \frac{\hat{\gamma}_{t,j}(\theta)}{\hat{\gamma}_t(\theta)} \right|.
    \end{align*}
    The remaining sum is $O(\theta^3)$ so it suffices to bound the contribution of terms of order $\geq 3$ in each summand. We have
   \begin{equation} \label{eg:gamma_quotient}
       \frac{\hat{\gamma}_{t,j}(\theta)}{\hat{\gamma}_t(\theta)} = \exp(ia\theta + b\theta),
   \end{equation}    
    where
    $$  a = m_{t,j}, \qquad b = \frac{1}{2}(v_t-u_{t,j}). $$
    By Proposition \ref{prop:mean_properties} and \ref{prop:v_close} we get $|a|<1$, and $|b| <9 $ so Lemma \ref{lem:order_3} shows that \eqref{eg:gamma_quotient} is bounded in absolute value by $\mathcal{K}(1,9,\theta_0)|\theta|^3 < \mathcal{K}_2(\theta_0) |\theta|^3 \leq \mathcal{K}_2(\theta_0) \occ{\tO\tl}{t} |\theta|^3$, where $\mathcal{K}_2(\theta_0)$ is the same as in Lemma \ref{lem:normal_approximation}. Letting $K(\theta_0) = \mathcal{K}_1(\theta_0)+\mathcal{K}_2(\theta_0)= 4\ell \mathcal{K}(3,19,\theta_0)$, we obtain the result.
    \end{proof}

\section{Upper bound on the characteristic function} \label{sec:upper}

In this section we prove the final ingredient of our main theorem, namely Proposition \ref{prop:char_fun_bound}. First, we state a standard bound, with proof included for completeness.

\begin{lemma} \label{lem:norm_bound}
    For any $\theta \in [-\pi,\pi]$ and $k \in \Z$ we have
    $$  |1+e(k\theta)| + |1+e((k+1)\theta)| \leq 4 - \frac{\theta^2}{\pi^2}. $$
\end{lemma}
\begin{proof}
    We have 
    $$  |1+e(k\theta)| + |1+e((k+1)\theta)| = 2\left( \left| \cos \frac{k\theta}{2} \right| + \left| \cos \frac{(k+1)\theta}{2} \right| \right). $$
    To bound this further, consider the function
    $$g(x) = \frac{x}{\sin x},$$
    which is strictly increasing on the interval $[0,\pi/2]$. Therefore, for any $x \in [-\pi/2,\pi/2]$ we have
    $$ \frac{x^2}{4 \sin^2 \frac{x}{2}} = g\left( \frac{x}{2}\right)^2 \leq g\left( \frac{\pi}{4}\right)^2 = \frac{\pi^2}{8}. $$
    Take any $y \in \R$ and put $x = \pi \| y \| \in [-\pi/2,\pi/2]$, where $\|y\|$ denotes the distance from $y$ to the nearest integer. We obtain
    $$|\cos \pi y| = \cos \pi \|y\| = 1-2 \sin^2 \frac{\pi \|y\|}{2} \leq 1-4\|y\|^2,$$
    and therefore
    \begin{align*}
        \left| \cos \frac{k\theta}{2} \right| + \left| \cos \frac{(k+1)\theta}{2} \right| &\leq 2- 4 \left(\left\| \frac{k\theta}{2\pi}\right\|^2+\left\| \frac{(k+1)\theta}{2\pi}\right\|^2\right)  \\
        &\leq 2 - 2 \left( \left\| \frac{k\theta}{2\pi}\right\|+\left\|\frac{k\theta}{2\pi} + \frac{\theta}{2\pi}\right\|\right)^2 \\
        &\leq 2 - 2\left\| \frac{\theta}{2\pi} \right\|^2 = 2 - \frac{\theta^2}{2\pi},
    \end{align*}
where the last inequality is a consequence of the inequality $\|x\|+\|x+y\| \geq \|y\|$ valid for all $x,y \in \R$. The result follows.
\end{proof}

We can now prove Proposition \ref{prop:char_fun_bound}.

% \begin{proposition} \label{prop:char_fun_bound}
%     For all $t \in \N$ we have
%     %where $x$ denotes any binary block of length $\ell-1$. then
%     $$ |\gamma_t(\theta)| \leq  \left(1- \frac{\theta^2}{2^{\ell+2}} \pi^2\right)^{K}, $$
%     where $K = \frac{2}{\ell+3}|t|_{01}-\frac{4}{3}$.
% \end{proposition}
\begin{proof}[Proof of Proposition \ref{prop:char_fun_bound}]
  Similarly as in the proof of Proposition \ref{prop:normal_approx}, we define block matrices
$$ \Lambda_t = \begin{bmatrix}
    \Gamma_t \\ \Gamma_{t+1}
\end{bmatrix}, \qquad D_{2t} = \begin{bmatrix}
    A_{2t} & \mathbf{0} \\ B_{2t+1} & C_{2t+1}
\end{bmatrix}, 
\qquad D_{2t+1} = \begin{bmatrix}
    B_{2t+1} & C_{2t+1} \\ \mathbf{0} & A_{2t+2}
\end{bmatrix},  $$
so that Proposition \ref{prop:char_fun_rec} yields
$$  \Lambda_{2t} = D_{2t} \Lambda_t, \qquad  \Lambda_{2t+1} = D_{2t+1} \Lambda_t. $$
Hence, if $t = [\varepsilon_m \cdots \varepsilon_1 \varepsilon_0]_2$, then
\begin{equation} \label{eq:matrix_prod}
    \Lambda_t = D_{[\varepsilon_{\ell-1} \cdots \varepsilon_0]_2} D_{[\varepsilon_{\ell} \cdots \varepsilon_1]_2} \cdots D_{[\varepsilon_{m+\ell-1}\cdots\varepsilon_{m}]_2}\Lambda_0, 
\end{equation} 
where we put $\varepsilon_j = \tO$ for $j > m$. Roughly speaking, each digit in the binary expansion corresponds to precisely one factor in the above product, although to determine the exact form of that factor $\ell-1$ preceding digits have to be known as well.

Since $|\gamma_t(\theta)| \leq \|\Lambda_t(\theta)\|_\infty$, it is sufficient to prove for $\occ{\tO\tl}{t} \geq \ell+3$ the inequality
\begin{equation} \label{eq:norm_ineq2}
     \|\Lambda_t(\theta)\|_\infty \leq  \left(1- \frac{\theta^2}{2^{\ell+2} \pi^2}\right)^{\occ{\tO\tl}{t}/(\ell+3)}.
\end{equation}

We split the reasoning into two parts, depending on whether $\ell=2$ or $\ell > 2$. Starting with $\ell=2$, we leave out some of the details since the reasoning is very similar as in \cite[Proposition 3.10]{SobolewskiSpiegelhofer2023}. Consider all possible strings $\uptau_3 \uptau_2 \tO\tl$, where $\uptau_2, \uptau_3 \in \{\tO,\tl\}$. There are at least $\fl{\occ{\tO\tl}{t}/2} > \occ{\tO\tl}{t}/5$ non-overlapping occurrences of these strings in the binary expansion of $t$ (because $\occ{\tO\tl}{t} \geq 5$). The  subproduct in \eqref{eq:matrix_prod} corresponding to each such occurrence is of the form
\begin{equation} \label{eq:subproduct}
    D_{[\tO\tl]_2} D_{[\uptau_2 \tO]_2} D_{[\uptau_3 \uptau_2]_2} D_{[\uptau_4 \uptau_3]_2},
\end{equation}
where $\uptau_4 \in \{\tO,\tl\}$ is the digit lying directly to the left of the string. By direct computation, for all $w$ of length $\ell=2$ and every choice of $\uptau_2,\uptau_3,\uptau_4$ each row of \eqref{eq:subproduct} has an entry containing the sub-expression $(1+\e(\theta))/16$ (up to multiplication by a power of $\e(\pm\theta)$). By Lemma \ref{lem:norm_bound} applied to $k=0$ we can thus bound the row-sum norm by $1 - \theta^2/(16\pi^2)$.
% $$  1 - \frac{1}{8} + \frac{1}{16}\left(2 - \frac{\theta^2}{\pi^2} \right) = 1 - \frac{\theta^2}{16\pi^2}. $$
    Using this once per each non-overlapping string $\uptau_3 \uptau_2 \tO\tl$, combined with submultiplicativity of $\|\cdot\|_{\infty}$ and the fact that $\|D_r(\theta)\|_{\infty} \leq 1$ for all $r$, we get \eqref{eq:norm_ineq2}.

In the general case $\ell \geq 3$ the idea is similar, however we need to rely on the properties of the matrices rather than direct computation. Let $m=\fl{(\ell+3)/2}$ and consider every $m$th occurrence of $\tO\tl$ in $(t)_2$, reading from left to right. Unlike in the rest of the paper, here we do not add a leading $\tO$ to the expansion. If such an occurrence is preceded by a $\tO$, then it lies in the ``middle'' of a string of the form $\tO^k\tO\tl u$, and otherwise in the ``middle'' of $\tl^k\tO\tl u$, where $u \in \{\tO,\tl\}^{\ell-1}$ and $k \geq 1$ is assumed to be maximal. Each of these strings contains at most $1+ (\ell-1)/2 \leq m-1$ occurrences of $\tO\tl$, hence they do not overlap. Since $\occ{\tO\tl}{t} \geq \ell+3$, the the number of non-overlapping strings of considered form is at least
%$ \fl{(|t|_{01}-1)/L} \geq |t|_{01}/L - 1$
$$  \fl{\frac{\occ{\tO\tl}{t}-1}{m}} \geq \frac{\occ{\tO\tl}{t}}{m}-1 \geq  \frac{2}{\ell+3} \occ{\tO\tl}{t} -1 \geq \frac{1}{\ell+3} \occ{\tO\tl}{t}.$$

We now show that each subproduct of \eqref{eq:matrix_prod} corresponding to one of these strings causes $\|\Lambda_t(\theta)\|_\infty$ to decrease (outside $\theta=0$).  Starting with the  case $\tO^k\tO\tl u$,
let $v \in \{\tO,\tl\}^{\ell-1}$ be the string lying directly to the left in the binary expansion (possibly with leading zeros) and write $v\tO^k\tO\tl u = \uptau_{n+\ell-1} \cdots \uptau_1\uptau_0$. It is possible that $v$ overlaps with the previous string, however this does not affect our reasoning. By the assumption that $k$ is maximal, $v$ must end with a $\tl$.  For $h \in \{0,1,\ldots,n\}$ we define $2^{\ell-1} \times 2^{\ell-1}$ matrices $E_h, F_h, G_h, J_h$ by writing in block form:
\begin{equation} \label{eq:matrix_prod_2}
\begin{bmatrix}
      E_{h} & F_{h} \\
      G_{h} & H_{h}
  \end{bmatrix}  = D_{[\uptau_{\ell-1} \cdots \uptau_0]_2} D_{[\uptau_{\ell} \cdots \uptau_1]_2} \cdots D_{[\uptau_{h+\ell-1}\cdots\uptau_{h}]_2}. 
  \end{equation}
% The corresponding part in the matrix product \eqref{eq:matrix_prod} is
% \begin{equation} \label{eq:matrix_prod_2}
%   \begin{bmatrix}
%       E_{n-\ell+1} & F_{n-\ell+1} \\
%       G_{n-\ell+1} & H_{n-\ell+1}
%   \end{bmatrix}  = D_{[\uptau_{\ell-1} \cdots \uptau_0]_2} D_{[\uptau_{\ell} \cdots \uptau_1]_2} \cdots D_{[\uptau_{n+\ell-1}\cdots\uptau_{n}]_2}. \end{equation}
We first examine these products evaluated at $\theta=0$ in order to gain some information about their nonzero entries. To simplify the notation, we let 
$$  \hat{D}_0 = D_{2t}(0) =\begin{bmatrix}
    A & \mathbf{0} \\ B & C
\end{bmatrix}, 
\qquad \hat{D}_1 = D_{2t+1}(0) = \begin{bmatrix}
    B & C \\ \mathbf{0} & A
\end{bmatrix}    $$
so that \eqref{eq:matrix_prod_2} evaluated at $\theta=0$ becomes
$$ \begin{bmatrix}
      \hat{E}_{h} & \hat{F}_{h} \\
      \hat{G}_{h} & \hat{H}_{h}
  \end{bmatrix} = \hat{D}_{\uptau_0} \hat{D}_{\uptau_1} \cdots \hat{D}_{\uptau_{h}}. $$
  To begin, consider the product of the factors corresponding to $u$, namely the case $h=\ell-2$.
  By a similar reasoning as in the proof of Proposition \ref{prop:v_close}, for each $j \in \{0,1,\ldots,2^{\ell-1}-1\}$ the $j$th row of precisely one of $\hat{E}_{\ell-2}, \hat{F}_{\ell-2}$  and precisely one of $\hat{G}_{\ell-2}, \hat{H}_{\ell-2}$ is nonzero. Moreover, the nonzero rows inside each block have all entries equal to $1/2^{\ell-1}$. Observe that for any $2^{\ell-1} \times 2^{\ell-1}$ matrix $M$ whose all columns are identical, we have $MA=M$ and $MB=MC = \frac{1}{2}M$. Since $\hat{E}_{\ell-2}, \hat{F}_{\ell-2}, \hat{G}_{\ell-2}, \hat{H}_{\ell-2}$ have this property, multiplying $\hat{D}_{\uptau_0} \hat{D}_{\uptau_1} \cdots \hat{D}_{\uptau_{\ell-2}}$ from the right by $\hat{D}_{\uptau_{\ell-1}}\hat{D}_{\uptau_{\ell}} = \hat{D}_1 \hat{D}_0$, we get
%    \begin{equation} \label{eq:efgh_l}
%     \begin{bmatrix}
%       \hat{E}_{\ell} & \hat{F}_{\ell} \\
%       \hat{G}_{\ell} & \hat{H}_{\ell}
%   \end{bmatrix} =  
%   \begin{bmatrix}
%       \hat{E}_{\ell-2} & \hat{F}_{\ell-2} \\
%       \hat{G}_{\ell-2} & \hat{H}_{\ell-2}
%   \end{bmatrix}  \begin{bmatrix}
%     AB & AC \\ B^2 & BC+CA
% \end{bmatrix} = \frac{1}{4} \begin{bmatrix}
%     2 \hat{E}_{\ell-2} + \hat{F}_{\ell-2} & 2\hat{E}_{\ell-2} + 3\hat{F}_{\ell-2} \\
%     2 \hat{G}_{\ell-2} +  \hat{H}_{\ell-2} & 2\hat{G}_{\ell-2} + 3\hat{H}_{\ell-2}
% \end{bmatrix}. 
% \end{equation}
  \begin{equation} \label{eq:efgh_l}
    \begin{bmatrix}
      \hat{E}_{\ell} & \hat{F}_{\ell} \\
      \hat{G}_{\ell} & \hat{H}_{\ell}
  \end{bmatrix} =  
  \begin{bmatrix}
      \hat{E}_{\ell-2} & \hat{F}_{\ell-2} \\
      \hat{G}_{\ell-2} & \hat{H}_{\ell-2}
  \end{bmatrix}  \begin{bmatrix}
    BA + CB & C^2 \\ AB & AC
\end{bmatrix} = \frac{1}{4} \begin{bmatrix}
    3 \hat{E}_{\ell-2} + 2 \hat{F}_{\ell-2} & \hat{E}_{\ell-2} + 2\hat{F}_{\ell-2} \\
    3 \hat{G}_{\ell-2} + 2 \hat{H}_{\ell-2} & \hat{G}_{\ell-2} + 2\hat{H}_{\ell-2}
\end{bmatrix}. 
\end{equation}
This matrix has all entries at least $1/2^{\ell+1}$.

Moving on, for $\ell+1 \leq h  \leq n$ we have $\hat{D}_{\uptau_1} \cdots \hat{D}_{\uptau_{h}}(0) = \hat{D}_0$, which means that
\begin{equation} \label{eq:efgh_h} 
\begin{bmatrix}
      \hat{E}_{h} & \hat{F}_{h} \\
      \hat{G}_{h} & \hat{H}_{h}
  \end{bmatrix} = 
  \begin{bmatrix}
      \hat{E}_{h-1} & \hat{F}_{h-1} \\
      \hat{G}_{h-1} & \hat{H}_{h-1}
  \end{bmatrix}  \begin{bmatrix}
    A & \mathbf{0} \\ B & C
\end{bmatrix}  =
\frac{1}{2}\begin{bmatrix}
      2\hat{E}_{h-1}+\hat{F}_{h-1} & \hat{F}_{h-1} \\
      2\hat{G}_{h-1}+\hat{H}_{h-1} & \hat{H}_{h-1}
  \end{bmatrix}
 \end{equation}
  so by induction we get $\hat{E}_{h} \geq \hat{E}_{\ell}$ and $\hat{G}_{h} \geq \hat{G}_{\ell}$ entry-wise. In particular, the key property is that all entries of $\hat{E}_{h}$ and $\hat{G}_{h}$ are still at least $1/2^{\ell+1}$.

We are ready to bound the row-sum norm of the product \eqref{eq:matrix_prod_2} for $h=n$, which corresponds to the full string $\tO^k\tO\tl u$. Put $s= [v\tO]_2$ and observe that  $s \not \equiv 0 \pmod{2^{\ell-1}}$, because $v$ ends with a $\tl$ and $\ell \geq 3$. Multiplying $D_{[\uptau_{\ell-1} \cdots \uptau_0]_2} D_{[\uptau_{\ell} \cdots \uptau_1]_2} \cdots D_{[\uptau_{n+\ell-2}\cdots\uptau_{n-1}]_2}$ from the right by $D_{s}$ yields
\begin{equation}
    \label{eq:matrix_prod_3}
     \begin{bmatrix}
    E_{n} & F_{n} \\
    G_{n} & H_{n}
\end{bmatrix} = \begin{bmatrix}
    E_{n-1} & F_{n-1} \\
    G_{n-1} & H_{n-1}
\end{bmatrix}  \begin{bmatrix}
    A_s & \mathbf{0} \\
    B_{s+1} & C_{s+1}
\end{bmatrix}  = \begin{bmatrix}
    E_{n-1}A_s + F_{n-1} B_{s+1} & F_{n-1} C_{s+1}\\
    G_{n-1}A_s + H_{n-1} B_{s+1} & H_{n-1} C_{s+1}
\end{bmatrix}.  
\end{equation} 
We now focus on bounding the row-sum norm of the first ``row'' of blocks, where 
% We have
% \begin{align*}
%     \|F_{n-1}(\theta)B_{s+1}(\theta)\|_{\infty} &\leq \|\hat{F}_{n-1} B\|_{\infty} = \frac{1}{2} \|\hat{F}_{n-1}  \|_{\infty}, \\
%     \|F_{n-1}(\theta)C_{s+1}(\theta)\|_{\infty} &\leq \|\hat{F}_{n-1} C\|_{\infty} = \frac{1}{2} \|\hat{F}_{n-1}  \|_{\infty}.
% \end{align*}
% is at most
% \begin{align*}
%     \| E_{n-\ell-1}(\theta)A_s(\theta)\|_{\infty} + \|F_{n-\ell-1}(\theta) B_{s+1}(\theta)\|_{\infty} + \|F_{n-\ell-1}(\theta)C_{s+1}(\theta)\|_{\infty} &\leq \\
%     \| E_{n-\ell-1}(\theta)A_s(\theta)\|_{\infty}+ \|\hat{F}_{n-\ell-1} B\|_{\infty} + \|\hat{F}_{n-\ell-1}C\|_{\infty} &=  \\ \| E_{n-\ell-1}(\theta)A_s(\theta)\|_{\infty} + 1- \|\hat{E}_{n-\ell-1}\|_{\infty}
% \end{align*}  
the term responsible for the decrease in the norm will be $E_{n-1}A_s$. Assume for the sake of simplicity that $[w]_2 < 2^{\ell-1}$ and $[w]_2$ is even, i.e.\ $w$ begins and ends with a $0$. Then we have
$$ A_s(\theta) = \frac{1}{2} \begin{bmatrix}
    & 0 & & & 0 & \\
    & \vdots & & & \vdots & \\
    & 0 & & & 0 & \\
\cdots    & \e(-\theta) & \cdots & & 1 &  \cdots \\
\cdots    & 1 & \cdots & & 1 & \cdots \\
    & 0 & & & 0 & \\
    & \vdots & & & \vdots & \\
    & 0 & & & 0 & \\
\end{bmatrix},  $$
where we focus on rows $[w]_2, [w]_2+1$ and columns $[w]_2/2,[w]_2/2+2^{\ell-2}$.
% the the entries of $A_s(\theta)$ in column $[w]_2/2$, rows $[w]_2, [w]_2+1$ are $\e(-\theta)/2, 1/2$, respectively, while all the other entries in that column are $0$. 
% Similarly, the nonzero entries in column $[w]_2/2+2^{\ell-2}$ lie in rows $[w]_2, [w]_2+1$ and are $1/2, 1/2$. Here we use the fact that $s \not\equiv 0,2^{\ell-1} \pmod{2^\ell}$. 
Switching the first or last digit of $w$ to $\tl$ (or both) only swaps the position of $\e(-\theta)$ and one of the other listed entries equal to $1$, which leads to a similar reasoning.

Choose a row number, say $j$, and let $p(\theta)$, $q(\theta)$ be the entries $E_{n-1}(\theta)$ lying in that row and columns $[w]_2,[w]_2+1$, respectively. Note that $2^n p(\theta)$ and $2^n q(\theta)$ are polynomials in $\e(\pm \theta)$ with integer coefficients, so we can write
$$  p(\theta) = \frac{1}{2^{n}} \sum_{a=1}^{2^n p(0)} \e( x_a \theta), \qquad q(\theta) = \frac{1}{2^{n}} \sum_{a=1}^{2^n q(0)} \e(y_a \theta), $$
where $x_a,y_a$ are some integers which may repeat for different $a$. Also, $p(0) = q(0)$, since these values lie in the same row of $\hat{E}_{n-1}$. 

Now, the entries in row $j$, columns $[w]_2/2$ and $[w]_2/2+2^{\ell-2}$ in $E_{n-1}(\theta)A_s(\theta)$ are 
$(\e(-\theta)p(\theta) + q(\theta))/2$ and $(p(\theta) + q(\theta))/2$, respectively.  Hence, their joint contribution  to the row-sum norm of $j$th row of $E_{n-1}(\theta)A_s(\theta)$ is
\begin{align*}
    \frac{1}{2}\left(|\e(-\theta)p(\theta) + q(\theta)| + |p(\theta) + q(\theta)|\right) &\leq  \frac{1}{2^{n+1}} \sum_{a=1}^{2^n p(0)} \left(|\e((x_a-1)\theta)+\e(y_a\theta)| + |\e(x_a\theta)+\e(y_a\theta)|  \right) \\    &\leq \frac{p(0)}{2}\left(4 - \frac{\theta^2}{\pi^2} \right), 
\end{align*} 
where we have applied Lemma \ref{lem:norm_bound} to each summand. 
%Let $r(q)$ denote the sum of all entries in the $j$th row of $E_{n-1}(\theta)A_s(\theta)+ F_{n-1}(\theta)A_{s+1}(\theta)$, including $p(\theta)$ and $q(\theta)$. S
The sum of all entries in each row of \eqref{eq:matrix_prod_3} evaluated at $\theta=0$ equals $1$. By the earlier part of the proof we have $p(0) \geq 1/2^{\ell+1}$ so the row-sum norm of the $j$th row of $\begin{bmatrix}
     E_{n-1}A_s + F_{n-1} B_{s+1} &  F_{n-1} C_{s+1} 
\end{bmatrix}$
 is
$$  1 - 2p(0) + \frac{p(0)}{2}\left(4 - \frac{\theta^2}{\pi^2} \right)  \leq 1 - \frac{\theta^2}{2^{\ell+2} \pi^2}. $$
The same reasoning works also for the bottom ``row'' of blocks in \eqref{eq:matrix_prod_3}, and therefore we get  
\begin{equation} \label{eq:matrix_prod_bound}
   \| D_{[\uptau_{\ell-1} \cdots \uptau_0]_2} D_{[\uptau_{\ell} \cdots \uptau_1]_2} \cdots D_{[\uptau_{n+\ell-1}\cdots\uptau_{n}]_2}\|_{\infty} \leq 1 - \frac{\theta^2}{2^{\ell+2} \pi^2}.
\end{equation}

If we consider an occurrence of the string $\tl^k\tO \tl  u$, the argument is very similar so we omit most of the details. Again, we let $v\tl^k\tO \tl u = \uptau_{n+\ell-1} \cdots \uptau_1\uptau_0$, where $v \in \{\tO,\tl\}^{\ell-1}$ ends with a $0$. Until equality \eqref{eq:efgh_l} the calculations are exactly the same. Then, instead of \eqref{eq:efgh_h} we get
$$ \begin{bmatrix}
      \hat{E}_{h} & \hat{F}_{h} \\
      \hat{G}_{h} & \hat{H}_{h}
  \end{bmatrix}  =
\frac{1}{2}\begin{bmatrix}
      \hat{E}_{h-1} & \hat{E}_{h-1} +2\hat{F}_{h-1} \\
      \hat{G}_{h-1} &  \hat{G}_{h-1}+2\hat{H}_{h-1}
  \end{bmatrix},   $$
and therefore all entries of $\hat{F}_{h}$ and $\hat{H}_{h}$ are at least $1/2^{\ell+1}$ for $h \geq \ell$. We again put $s = [v\tl]_2$ and have $s+1 \not \equiv 0 \pmod{2^{\ell-1}}$. Further, we have
$$
     \begin{bmatrix}
    E_{n} & F_{n} \\
    G_{n} & H_{n}
\end{bmatrix} = \begin{bmatrix}
    E_{n-1} & F_{n-1} \\
    G_{n-1} & H_{n-1}
\end{bmatrix}  \begin{bmatrix}
    B_s & C_s \\
   \mathbf{0} & A_{s+1}
\end{bmatrix}  = \begin{bmatrix}
    E_{n-1}B_s & E_{n-1}C_s + F_{n-1} A_{s+1} \\
    G_{n-1}B_s & G_{n-1}C_s + H_{n-1} A_{s+1}
\end{bmatrix}.  
$$
This time, the crucial part is to bound $\| F_{n-1}(\theta) A_{s+1}(\theta)  \|_{\infty}$ and $\| H_{n-1}(\theta) A_{s+1}(\theta)  \|_{\infty}$, which is done in the same way as before. We thus again reach the bound 
\eqref{eq:matrix_prod_bound}. 

Inequality \eqref{eq:norm_ineq2} follows and the proof is finished.
% Using this inequality once per each non-overlapping string $0^k01u$ or $1^k01u$, submultiplicativity of $\|\cdot\|_{\infty}$ and the fact that $\|D_r(\theta)\| \leq 1$ for all $\theta$, we get \eqref{eq:norm_ineq2}.
% rows $[w]_2, [w]_2+1$ and columns $[w]_2/2, [w]_2/2 +2^{\ell-2}$ are
% $$   a_{s}([w]_2,[w]_2/2) = \frac{1}{2}\e(-\theta), \qquad a_{s}([w]_2,[w]_2/2 +2^{\ell-2}) = a_{s}([w]_2+1,[w]_2/2 ) = a_{s}([w]_2+1,[w]_2/2 +2^{\ell-2}) = \frac{1}{2} $$
% \begin{bmatrix}
%     B & C \\ 0 & A
% \end{bmatrix}  \begin{bmatrix}
%     A & 0 \\ B & C
% \end{bmatrix}
\end{proof}

% \section{Further remarks}

% \textcolor{red}{jakies ogolne warunki na $2$-regular?}

\subsection*{Acknowledgements}
The research was supported by the grant of the National Science Centre (NCN), Poland, no.\ UMO-2020/37/N/ST1/02655 and the grant of the Austrian Science Fund (FWF), no.\ P36137-N.

\bibliographystyle{amsplain}
\bibliography{references}

\providecommand{\bysame}{\leavevmode\hbox to3em{\hrulefill}\thinspace}
\providecommand{\MR}{\relax\ifhmode\unskip\space\fi MR }
\providecommand{\MRhref}[2]{%
  \href{http://www.ams.org/mathscinet-getitem?mr=#1}{#2}
}
\providecommand{\href}[2]{#2}
\begin{thebibliography}{10}

\bibitem{AlloucheShallit2003}
Jean-Paul Allouche and Jeffrey Shallit, \emph{The ring of {$k$}-regular
  sequences. {II}}, Theoret. Comput. Sci. \textbf{307} (2003), no.~1, 3--29.

\bibitem{Besineau1972}
Jean B\'esineau, \emph{Ind\'ependance statistique d'ensembles li\'es \`a la
  fonction ``somme des chiffres''}, Acta Arith. \textbf{20} (1972), 401--416.

\bibitem{DrmotaKauersSpiegelhofer2016}
Michael Drmota, Manuel Kauers, and Lukas Spiegelhofer, \emph{On a {C}onjecture
  of {C}usick {C}oncerning the {S}um of {D}igits of {$n$} and {$n+t$}}, SIAM J.
  Discrete Math. \textbf{30} (2016), no.~2, 621--649.

\bibitem{EmmeHubert2019}
Jordan Emme and Pascal Hubert, \emph{Central limit theorem for probability
  measures defined by sum-of-digits function in base 2}, Annali della {S}cuola
  {N}ormale {S}uperiore di {P}isa \textbf{XIX} (2019), no.~2, 757--780.

\bibitem{EmmePrikhodko2017}
Jordan Emme and Alexander Prikhod'ko, \emph{On the {A}symptotic {B}ehavior of
  {D}ensity of {S}ets {D}efined by {S}um-of-digits {F}unction in {B}ase 2},
  Integers \textbf{17} (2017), A58.

\bibitem{HostenJanvresseRue2024}
Yohan Hosten, \'{E}lise Janvresse, and Thierry de~la Rue, \emph{A central limit
  theorem for the variation of the sum of digits}, Ann. Inst. Henri
  Poincar\'{e} Probab. Stat. \textbf{60} (2024), no.~2, 1125--1149.

\bibitem{Legendre1830}
Adrien-Marie Legendre, \emph{Th\'{e}orie des nombres}, Firmin Didot fr\`{e}res,
  Paris, 1830.

\bibitem{SobolewskiSpiegelhofer2023}
Bartosz Sobolewski and Lukas Spiegelhofer, \emph{Block occurrences in the
  binary expansion}, 2023, Preprint, https://arxiv.org/abs/2309.00142.

\bibitem{Spiegelhofer2019}
Lukas Spiegelhofer, \emph{Approaching {C}usick's conjecture on the
  sum-of-digits function}, Integers \textbf{19} (2019), Paper No. A53.

\bibitem{Spiegelhofer2022}
\bysame, \emph{A lower bound for {C}usick's conjecture on the digits of {$n +
  t$}}, Math. Proc. Cambridge Philos. Soc. \textbf{172} (2022), no.~1,
  139--161.

\bibitem{SpiegelhoferWallner2019}
Lukas Spiegelhofer and Michael Wallner, \emph{The {T}u--{D}eng conjecture holds
  almost surely}, Electron. J. Combin. \textbf{26} (2019), no.~1, Paper 1.28.

\bibitem{SpiegelhoferWallner2023}
\bysame, \emph{The binary digits of $n+t$}, Ann. Sc. Norm. Super. Pisa, Cl.
  Sci. (5) \textbf{XXIV} (2023), no.~1, 1--31.

\bibitem{TuDeng2011}
Ziran Tu and Yingpu Deng, \emph{A conjecture about binary strings and its
  applications on constructing {B}oolean functions with optimal algebraic
  immunity}, Des. Codes Cryptogr. \textbf{60} (2011), no.~1, 1--14.

\end{thebibliography}

\end{document}